\documentclass[12pt]{amsart}

\usepackage{mathptmx} 
\DeclareSymbolFont{cmletters}{OML}{cmm}{m}{it}              
\DeclareSymbolFont{cmsymbols}{OMS}{cmsy}{m}{n}
\DeclareSymbolFont{cmlargesymbols}{OMX}{cmex}{m}{n}
\DeclareMathSymbol{\myjmath}{\mathord}{cmletters}{"7C}     \let\jmath\myjmath %Defining the missing commands: \jmath, \amalg and \coprod
\DeclareMathSymbol{\myamalg}{\mathbin}{cmsymbols}{"71}     \let\amalg\myamalg
\DeclareMathSymbol{\mycoprod}{\mathop}{cmlargesymbols}{"60}\let\coprod\mycoprod
\DeclareMathSymbol{\myalpha}{\mathord}{cmletters}{"0B}     \let\alpha\myalpha %Greek letters from Computer Modern
\DeclareMathSymbol{\mybeta}{\mathord}{cmletters}{"0C}      \let\beta\mybeta
\DeclareMathSymbol{\mygamma}{\mathord}{cmletters}{"0D}     \let\gamma\mygamma
\DeclareMathSymbol{\mydelta}{\mathord}{cmletters}{"0E}     \let\delta\mydelta
\DeclareMathSymbol{\myepsilon}{\mathord}{cmletters}{"0F}   \let\epsilon\myepsilon
\DeclareMathSymbol{\myzeta}{\mathord}{cmletters}{"10}      \let\zeta\myzeta
\DeclareMathSymbol{\myeta}{\mathord}{cmletters}{"11}       \let\eta\myeta
\DeclareMathSymbol{\mytheta}{\mathord}{cmletters}{"12}     \let\theta\mytheta
\DeclareMathSymbol{\myiota}{\mathord}{cmletters}{"13}      \let\iota\myiota
\DeclareMathSymbol{\mykappa}{\mathord}{cmletters}{"14}     \let\kappa\mykappa
\DeclareMathSymbol{\mylambda}{\mathord}{cmletters}{"15}    \let\lambda\mylambda
\DeclareMathSymbol{\mymu}{\mathord}{cmletters}{"16}        \let\mu\mymu
\DeclareMathSymbol{\mynu}{\mathord}{cmletters}{"17}        \let\nu\mynu
\DeclareMathSymbol{\myxi}{\mathord}{cmletters}{"18}        \let\xi\myxi
\DeclareMathSymbol{\mypi}{\mathord}{cmletters}{"19}        \let\pi\mypi
\DeclareMathSymbol{\myrho}{\mathord}{cmletters}{"1A}       \let\rho\myrho
\DeclareMathSymbol{\mysigma}{\mathord}{cmletters}{"1B}     \let\sigma\mysigma
\DeclareMathSymbol{\mytau}{\mathord}{cmletters}{"1C}       \let\tau\mytau
\DeclareMathSymbol{\myupsilon}{\mathord}{cmletters}{"1D}   \let\upsilon\myupsilon
\DeclareMathSymbol{\myphi}{\mathord}{cmletters}{"1E}       \let\phi\myphi
\DeclareMathSymbol{\mychi}{\mathord}{cmletters}{"1F}       \let\chi\mychi
\DeclareMathSymbol{\mypsi}{\mathord}{cmletters}{"20}       \let\psi\mypsi
\DeclareMathSymbol{\myomega}{\mathord}{cmletters}{"21}     \let\omega\myomega
\DeclareMathSymbol{\myvarepsilon}{\mathord}{cmletters}{"22}\let\varepsilon\myvarepsilon
\DeclareMathSymbol{\myvartheta}{\mathord}{cmletters}{"23}  \let\vartheta\myvartheta
\DeclareMathSymbol{\myvarpi}{\mathord}{cmletters}{"24}     \let\varpi\myvarpi
\DeclareMathSymbol{\myvarrho}{\mathord}{cmletters}{"25}    \let\varrho\myvarrho
\DeclareMathSymbol{\myvarsigma}{\mathord}{cmletters}{"26}  \let\varsigma\myvarsigma
\DeclareMathSymbol{\myvarphi}{\mathord}{cmletters}{"27}    \let\varphi\myvarphi

\usepackage{amsthm,amsmath,amssymb,amscd,graphics,enumerate,latexsym,verbatim,graphicx,color}
\usepackage[all]{xy}%\CompileMatrices\CompilePrefix{xymatrix/diagram}
\usepackage{tikz}

\usetikzlibrary{matrix,arrows}

\theoremstyle{plain}
\newtheorem{thm}{Theorem}[section]
\newtheorem{cor}[thm]{Corollary}
\newtheorem{lemma}[thm]{Lemma}

\theoremstyle{definition}
\newtheorem{df}[thm]{Definition}

\newtheorem{rem}[thm]{Remark}
\newtheorem{ex}[thm]{Example}

\newcommand{\arst}{s\stackrel{a}\longrightarrow t}
\newcommand{\ses}[3]{0\rightarrow #1\rightarrow #2\rightarrow#3\rightarrow 0}
\DeclareMathOperator{\Mat}{Mat}

\DeclareMathOperator{\Gr}{Gr}

\DeclareMathOperator{\Hom}{Hom}
\DeclareMathOperator{\Ext}{Ext}

\DeclareMathOperator{\Rel}{Rel}

\def\0{{\bf 0}}
\def\A{{\mathbb A}}

\def\C{{\mathbb C}}
\def\F{{\mathbb F}}

\def\N{{\mathbb N}}
\def\P{{\mathbb P}}

\def\Z{{\mathbb Z}}

\def\cB{{\mathcal B}}

\def\cW{{\mathcal W}}

\def\Fun{{\F_1}}

\def\Gm{\mathbb G_m}

\def\int{\textup{int}}
\def\ad{\textup{ad}}

\def\1{\textbf{1}}

\def\ud{{\underline{d}}}
\def\ue{{\underline{e}}}

\def\udim{{\underline{\dim}\, }}

\def\min{{\textup{min\,}}}

\def\={\equiv}
\def\n={\equiv\hspace{-10,5pt}/\hspace{3,5pt}}

\newdir{ >}{{}*!/-5pt/@^{(}}\newcommand{\arincl}[1]{\ar@{ >->}@<-0,0ex>#1} %inclusion arrow for xy-matrix with better spacing

\newcommand{\tinymat}[4]{\bigl( \begin{smallmatrix} #1 & #2 \\ #3 & #4 \end{smallmatrix} \bigr)}

\makeatletter
\let\@wraptoccontribs\wraptoccontribs
\makeatother

\begin{document}

\title[Schubert decompositions for quiver Grassmannians of tree modules]{Schubert decompositions for quiver Grassmannians\\ of tree modules}
\author{Oliver Lorscheid}
\date{}
\address{IMPA, Estrada Dona Castorina 110, 22460-320 Rio de Janeiro, Brazil}
\email{lorschei@impa.br}
\contrib[with an appendix by]{Thorsten Weist}

\begin{abstract}
 Let $Q$ be a quiver, $M$ a representation of $Q$ with an ordered basis $\cB$ and $\ue$ a dimension vector for $Q$. In this note we extend the methods of \cite{L12} to establish Schubert decompositions of quiver Grassmannians $\Gr_\ue(M)$ into affine spaces to the ramified case, i.e.\ the canonical morphism $F:T\to Q$ from the coefficient quiver $T$ of $M$ w.r.t.\ $\cB$ is not necessarily unramified.

 In particular, we determine the Euler characteristic of $\Gr_\ue(M)$ as the number of \emph{extremal successor closed subsets of $T_0$}, which extends the results of Cerulli Irelli (\cite{Cerulli11}) and Haupt (\cite{Haupt12}) (under certain additional assumptions on $\cB$). 
\end{abstract}

\maketitle

%%%%%%%%%%%%%%%%%%%%%%%%%%%%%%%%%%%%%%%%%%%%%%%%%%%%%%%%%%%%%%%%%%%%%%%%%%%%%%%%%%%%%%%%%%%%%%%%%%%%%%%%%%%%%%%%%%%%%%%%%%%%%%%%%%%%%%%%%%%%%%%%%%%%%%%%%%%%%%%%%%%%%%%%%%%%%%%%%%%%%%%%%%%%%%%%%%%%%%%%%%%%%%%%%%%%%%%%%%%%
%%%%%%%%%%%%%%%%%%%%%%%%%%%%%%%%%%%%%%%%%%%%%%%%%%%%%%%%%%%%%%%%%%%%%%%%%%%%%%%%%%%%%%%%%%%%%%%%%%%%%%%%%%%%%%%%%%%%%%%%%%%%%%%%%%%%%%%%%%%%%%%%%%%%%%%%%%%%%%%%%%%%%%%%%%%%%%%%%%%%%%%%%%%%%%%%%%%%%%%%%%%%%%%%%%%%%%%%%%%%

\section*{Introduction}
\label{section: introduction}
 
 \noindent
 The recent interest in quiver Grassmannians stems from a formula of Caldero and Chapoton (\cite{Caldero-Chapoton06}) that relates cluster variables of a quiver $Q$ with the Euler characteristics of the quiver Grassmannians of exceptional modules of $Q$. Formulas for the Euler characteristics for a given quiver yields a description of the associated cluster algebra in terms of generators and relations. This opened a way to understand cluster algebras, which are defined by an infinite recursive procedure, in terms of closed formulas---provided one knows the Euler characteristics of the associated quiver Grassmannians.

\subsection*{Torus actions and cluster algebras associated with string algebras}
 While the classification of all cluster algebras seems to be as much out of reach as a classification of wild algebras, their is some hope to understand and classify cluster algebras that are associated with tame algebras. A first step in this direction has been realized by Cerulli-Irelli (\cite{Cerulli11}) and Haupt (\cite{Haupt12}) who established a formula for the Euler characteristics of quiver Grassmannians in the so-called unramified case. These results sufficed to understand all cluster algebras associated path with string algebras.

 We review the method of Cerulli-Irelli and Haupt in brevity: following Ringel (\cite{Ringel98}), every exceptional representation $M$ of a quiver $Q$ has tree basis $\cB$, i.e.\ the coefficient quiver $T=\Gamma(M,\cB)$ is a tree. A subset $\beta$ of $T_0=\cB$ is \emph{successor closed} if for all $i\in\beta$ and all arrows $\alpha:i\to j$ in $T$, also $j\in\beta$. A subset $\beta$ of $T_0=\cB$ is \emph{of type $\ue=(e_p)_{p\in Q_0}$} if $\#\beta\cap M_p=e_p$ for all $p\in Q_0$.

 If the canonical morphism $F:T\to Q$ is \emph{unramified}, i.e.\ the morphism of the underlying CW-complexes is locally injective, then one can define a (piecewise continuous) action of the torus $\Gm$ on $\Gr_\ue(M)$ that has only finitely many fixed points. This yields the formula
 \[
  \chi\bigl( \Gr_\ue(M) \bigr) \quad = \quad \# \ \{\,\text{fixed points}\,\} \quad = \quad \# \ \left\{\,\text{successor closed }\beta\subset T_0 \text{ of type }\ue \,\right\}.
 \]

 For other types of cluster algebras, the exceptional modules are in general not unramified tree modules. This is, for instance, the case of cluster algebras associated with exceptional Dynkin quivers of types $\widetilde D$ and $\widetilde E$, or, more general, for cluster algebras associated with clannish algebras or exceptional tame algebras. Therefore other methods are required to treat ramified tree modules.

\subsection*{Cluster algebras from marked surfaces}
 Fomin, Shapiro and Thurston explore in \cite{FST08} the connection between cluster algebras and marked surfaces. Namely to each surface with boundary and finitely many marked points such that each boundary component contains at least one marked point, one can associate a cluster algebra. In \cite{FST08}, it is shown that all cluster algebras associated with quivers of extended Dynkin types $\widetilde A$ and $\widetilde D$ come from marked surfaces.

 This connection with marked surfaces yields a description of the cluster variables in terms of triangulations of the surface, which leads to a combinatorial description of the algebra. For unpunctured surfaces, i.e.\ all marked points are contained in the boundary, Musiker, Schiffler and Williams construct in \cite{MSW13} a basis for the associated cluster algebra.

 Cluster algebras of punctured surfaces, which includes type $D$ algebras, are more difficult to treat since not all mutations of clusters come from flips of triangulations. For recent results in this direction, see Qiu and Zhou's paper \cite{Qiu-Zhou13}. However, these methods do not suffice yet for a complete understanding of the associated cluster algebras.

\subsection*{Schubert decompositions and ramified tree modules}
 Caldero and Reineke (\cite{Caldero-Reineke08}) show that $\Gr_\ue(M)$ is smooth projective if $M$ is exceptional. If $M$ is an equioriented string module, i.e.\ the coefficient quiver $T$ is an equioriented Dynkin quiver of type $A_n$, then $\Gr_\ue(M)$ has a continuous torus action with finitely many fixed points, see \cite{Cerulli11}. Thus if $M$ is an exceptional equioriented string module, then the Bia{\l}ynicki-Birula decomposition yields a decomposition of $\Gr_\ue(M)$ into affine spaces, cf.\ \cite[Thm.\ 4.3]{BB73}. 

 While a torus action with finitely many fixed points determines the Euler characteristic, a decomposition of $\Gr_\ue(M)$ into affine spaces determines the (additive structure of the) cohomology of $\Gr_\ue(M)$, which is a much stronger result. In particular, we re-obtain the Euler characteristic as the number of affine spaces occurring in the decomposition. However, the class of exceptional equioriented string modules is very limited. In particular, most exceptional modules of affine type $D$ are not of this kind. 

 In the author's paper \cite{L12}, we extend decompositions of $\Gr_\ue(M)$ into affine spaces to a larger class of quiver Grassmannians by a different method. Namely, the choice of an ordered basis $\cB$ of $M$ defines a decomposition of $\Gr_\ue(M)$ into Schubert cells, which are, in general, merely closed subsets of affine spaces. In certain cases, however, these Schubert cells are affine spaces themselves. The method of proof is to exhibit explicit presentations of Schubert cells in terms of generators and relations. 
 
 One requirement of \cite{L12} is that the morphism $F:T\to Q$ is unramified. It is the purpose of this note to extend the methods of \cite{L12} to ramified $F:T\to Q$. In particular, this extends, under the given additional assumptions, the formula of Cerulli-Irelli and Haupt to the ramified case. 

 As will be shown in the joint work \cite{Lorscheid-Weist} with Thorsten Weist, the results of this text are indeed applicable to all exceptional modules of affine type $\tilde D_n$, which yields combinatorial formulas for the Euler characteristics of $\Gr_\ue(M)$.

\subsection*{The main result of this text}
 An arrow $\alpha$ of $T$ \emph{extremal} if for every other arrow $\alpha'$ with $F(\alpha')=F(\alpha)$ either $s(\alpha) <s(\alpha')$ or $t(\alpha') <t(\alpha)$. A subset $\beta$ of $T_0$ \emph{extremal successor closed} if for every $i\in\beta$ and every extremal arrow $\alpha:i\to j$ in $T$, also $j\in\beta$.

 Under certain additional assumptions on $\cB$, the quiver Grassmannian $\Gr_\ue(M)$ decomposes into affine spaces (Theorem \ref{thm: main theorem}), and the parametrization of the non-empty Schubert cells yields the formula
 \[
  \chi\bigl( \Gr_\ue(M) \bigr) \quad = \quad \# \ \left\{\,\text{extremal successor closed }\beta\subset T_0 \text{ of type }\ue \,\right\}
 \]
 (Corollary \ref{cor: formula for the euler characteristic}).

\subsection*{Content overview}
 To keep the technical complexity as low as possible, we restrict ourselves in this text to tree modules over the complex numbers, though the methods work in the more general context of modules of tree extensions over arbitrary rings as considered in \cite{L12}. The technique of proof in the ramified case is essentially the same as the one used in \cite{L12}. But since the presentation of our results is different and simplified, we include all details.

 This text is organized as follows. In Section \ref{section: setup}, we review basic facts about quiver Grassmannians, their Schubert decompositions and tree modules. In Section \ref{section: presentation of schubert cells}, we describe generators and relations for a Schubert cell, which are labeled by \emph{relevant pairs} and \emph{relevant triples}, respectively. In Section \ref{section: preliminaries}, we introduce \emph{extremal successor closed subsets}, \emph{polarizations} and \emph{maximal relevant pairs}, and we establish preliminary facts. In Section \ref{section: schubert decompositions for tree modules}, we state the main results and conclude with several remarks and examples.

 In Appendix \ref{appendix} (by Thorsten Weist), we show how to establish polarizations for exceptional modules along Schofield induction.

\subsection*{Acknowledgements} 
 I would like to thank Dave Anderson, Giovanni Cerulli Irelli, Markus Reineke, Cec\'ilia Salgado and Jan Schroer for helpful discussions. I would like to thank Thorsten Weist for including his ideas on polarizations as an appendix to this text.

\section{Setup}
\label{section: setup}

 \noindent
To start with, let us explain the notation and terminology that we use in this text. By a variety we understand the space of complex points of an underlying scheme, and we broadly ignore the schematic structure of quiver Grassmannians. For more details on the notions in this section, see Sections 1 and 2 of \cite{L12}.

\subsection{Quiver Grassmannians}\label{subsection: quiver Grassmannians}

Let $Q=(Q_0,Q_1,s,t)$ be a quiver, $M=\bigl(\{M_i\}_{i\in Q_0},\{M_\alpha\}_{\alpha\in Q_1}\bigr)$ a (complex) representation of $Q$ with dimension vector $\ud=\udim M$ and $\ue\leq\ud$ another dimension vector for $Q$. The \emph{quiver Grassmannian $\Gr_\ue(M)$} is the set of subrepresentations $N$ of $M$ with $\udim N=\ue$. A \emph{basis $\cB$ for $M$} is the union $\bigcup_{p\in Q_0} \cB_p$ of bases $\cB_p$ for the vector spaces $M_p$. An \emph{ordered basis of $M$} is a basis $\cB$ of $M$ whose elements $b_1,\dotsc,b_n$ are linearly ordered. The  choice of an ordered basis yields an inclusion
\[
 \Gr_\ue(M) \quad \longrightarrow \quad \prod_{p\in Q_0} \ \Gr(e_p,d_p),
% \quad \stackrel{\text{Pl\"ucker embedding}}\longrightarrow \quad \prod_{p\in Q_0} \ \P\bigl(\Lambda^{e_p}\C^{d_p}\bigr).
\]
that sends $N$ to $(N_p)_{p\in Q_0}$, which endows $\Gr_\ue(M)$ with the structure of a projective variety.

\subsection{Schubert decompositions}\label{subsection: schubert decompositions}

A point of the Grassmannian $\Gr(e,d)$ is an $e$-dimensional subspace $V$ of $\C^{d}$. Let $V$ be spanned by vectors $w_1,\dotsc,w_e\in\C^d$. We write $w=(w_{i,j})_{i=1\dotsc d,j=1\dotsc, e}$ for the matrix of all coordinates of $w_1,\dotsc,w_e$. The Pl\"ucker coordinates 
\[
 \Delta_I(V) \quad = \quad \det(w_{i,j})_{i\in I,j=1\dotsc, e}
\]
(where $I$ is a subset of $\{1,\dotsc,d\}$ of cardinality $e$) define a point $(\Delta_I(V))_I$ in $\P\bigl(\Lambda^{e}\C^{d}\bigr)$. For two ordered subsets $I=\{i_1,\dotsc,i_e\}$ and $J=\{j_1,\dotsc,j_e\}$ of $\{1,\dotsc,d\}$, we define $I\leq J$ if $i_l\leq j_l$ for all $l=1\dotsc,e$. The Schubert cell $C_I(d)$ of $\Gr(e,d)$ is defined as the locally closed subvariety of all subspaces $V$ such that $\Delta_I(V)\neq 0$ and $\Delta_J(V)=0$ for all $J>I$.

Given a quiver $Q$, a representation $M$ with ordered basis $\cB$ and a dimension vector $\ue$, we say that a subset $\beta$ of $\cB$ is \emph{of type $\ue$} if $\beta_p=\beta\cap\cB_p$ is of cardinality $e_p$ for every $p\in Q_0$. For $\ud=\udim M$, the Schubert cell $C_\beta(\ud)$ is defined as the locally closed subset $\prod_{p\in Q_0} C_{\beta_p}(d_p)$ of $\prod_{p\in Q_0} \ \Gr(e_p,d_p)$. The \emph{Schubert cell $C_\beta^M$} is defined as the intersection of $C_\beta(\ud)$ with $\Gr_\ue(M)$ inside $\prod_{p\in Q_0} C_{\beta_p}(d_p)$. The \emph{Schubert decomposition of $\Gr_\ue(M)$ (w.r.t.\ the ordered basis $\cB$)} is the decomposition
\[
 \Gr_\ue(M) \quad = \quad \coprod_{\substack{\beta\subset\cB\\\text{of type }\ue}} \ C_\beta^M
\]
into locally closed subvarieties. Note that the Schubert cells $C_\beta^M$ are affine varieties, but that they are, in general, not affine spaces. In particular, a Schubert cell $C_\beta^M$ might be empty. We say that $\Gr_\ue(M) = \coprod \ C_\beta^M$ is a \emph{decomposition into affine spaces} if every Schubert cell $C_\beta^M$ is either an affine space or empty.

\subsection{Tree modules}\label{subsection: tree modules}

Let $M$ be a representation of a quiver $Q$ with basis $\cB$. Let $\alpha:s\to t$ be an arrow of $Q$ and $b\in\cB_s$. Then we have the equations 
\[
 M_\alpha(b) \quad = \quad \sum_{b'\in\cB_t} \lambda_{\alpha,b,c} \, c 
\]
with uniquely determined coefficients $\lambda_{\alpha,b,c}\in\C$. The \emph{coefficient quiver of $M$ w.r.t.\ $\cB$} is the quiver $T=\Gamma(M,\cB)$ with vertex set $T_0=\cB$ and with arrow set
\[
 T_1 \quad = \quad \bigl\{ \ (\alpha,b,c)\in Q_1\times\cB\times\cB \ \bigl| \ b\in\cB_{s(\alpha)}, c\in\cB_{t(\alpha)}\text{ and }\lambda_{\alpha,b,c}\neq 0 \ \bigr\}.
\]
It comes together with a morphism $F:T\to Q$ that sends $b\in\cB_p$ to $p$ and $(\alpha,b,c)$ to $\alpha$, and with a thin sincere representation $N=N(M,\cB)$ of $T$ with basis $\cB$ and $1\times 1$-matrices $N_{(\alpha,b,c)}=(\lambda_{\alpha,b,c})$. Note that $M$ is canonically isomorphic to the push-forward $F_\ast N$ (cf.\ \cite[Section 4]{L12}).

The representation $M$ is called a \emph{tree module} if there exists a basis $\cB$ of $M$ such that the coefficient quiver $T=\Gamma(M,\cB)$ is a tree. We call such a basis a \emph{tree basis for $M$}.

Note that if $T$ is a tree, then we can replace the basis elements $b$ by certain non-zero multiples $b'$ such that all $\lambda_{\alpha,b,c}$ equal $1$. We refer to this assumption by the expression $M=F_\ast T$ where we identify $T$, by abuse of notation, with its thin sincere representation with basis $T_0=\cB$ and matrices $(1)$. In this case, $M$ and $\cB$ are determined as the push-forward of this thin sincere representation of $T$ along $F:T\to Q$. Note that $T$ is in general not determined by $M$: there are examples of tree modules $M$ and bases $\cB$ and $\cB'$ such that $\Gamma(M,\cB)$ and $\Gamma(M,\cB')$ are non-isomorphic trees.

%%%%%%%%%%%%%%%%%%%%%%%%%%%%%%%%%%%%%%%%%%%%%%%%%%%%%%%%%%%%%%%%%%%%%%%%%%%%%%%%%%%%%%%%%%%%%%%%%%%%%%%%%%%%%%%%%%%%%%%%%%%%%%%%%%%%%%%%%%%%%%%%%%%%%%%%%%%%%%%%%%%%%%%%%%%%%%%%%%%%%%%%%%%%%%%%%%%%%%%%%%%%%%%%%%%%%%%%%%%%
%%%%%%%%%%%%%%%%%%%%%%%%%%%%%%%%%%%%%%%%%%%%%%%%%%%%%%%%%%%%%%%%%%%%%%%%%%%%%%%%%%%%%%%%%%%%%%%%%%%%%%%%%%%%%%%%%%%%%%%%%%%%%%%%%%%%%%%%%%%%%%%%%%%%%%%%%%%%%%%%%%%%%%%%%%%%%%%%%%%%%%%%%%%%%%%%%%%%%%%%%%%%%%%%%%%%%%%%%%%%

\section{Presentations of Schubert cells}
\label{section: presentation of schubert cells}

\noindent
Let $Q$ be a quiver and $M$ a representation with ordered basis $\cB$ and dimension vector $\ud$. Let $\ue$ be another dimension vector for $Q$ and $\beta\subset\cB$ of type $\ue$. In this section, we will describe coordinates and relations for the Schubert cell $C_\beta^M$ of $\Gr_\ue(M)$.

\subsection{Normal form for matrix representations}\label{subsection: normal form of matrix representations}

Let $N$ be a point of $C_\beta^M$. Then $N_p$ is a $e_p$-dimensional subspace of $M_p$ for every $p\in Q_0$ and has a basis $(w_j)_{j\in\beta_p}$ where $w_j=(w_{i,j})_{i\in \cB_p}$ are column vectors in $M_p$. If we define $w_{i,j}=0$ for $i,j\in\cB$ whenever $j\in\beta$, or $i\in\cB_p$ and $j\in\cB_q$ with $p\neq q$, then we obtain a matrix $w=(w_{i,j})_{i,j\in\cB}$.
We call such a matrix $w$ a \emph{matrix representation of $N$}. Note that $N$ is determined by the matrix representation $w$, but there are in general many different matrix representations of $N$.

We say that a matrix $w=(w_{i,j})_{i,j\in\cB}$ in $\Mat_{\cB\times\cB}$ is \emph{in $\beta$-normal form} if 
 \begin{enumerate}
  \item\label{part1} $w_{i,i}=1$ for all $i\in\beta$,
  \item\label{part2} $w_{i,j}=0$ for all $i,j\in\beta$ with $j\neq i$,
  \item\label{part3} $w_{i,j}=0$ for all $i\in\cB$ and $j\in\beta$ with $j<i$,
  \item\label{part4} $w_{i,j}=0$ for all $i\in\cB$ and $j\in\cB-\beta$, and
  \item\label{part5} $w_{i,j}=0$ for all $i\in\cB_p$ and $j\in\beta_q$ with $p\neq q$.
 \end{enumerate}

\begin{lemma} \label{lemma: normal form}
 Every $N\in C_\beta^M$ has a unique matrix representation $w=(w_{i,j})_{i,j\in\cB}$ in $\beta$-normal form.
\end{lemma}

\begin{proof}
 The uniqueness follows from the fact that a matrix $w$ in $\beta$-normal form is in reduced column echelon form by \eqref{part1}--\eqref{part4}. The vanishing of the Pl\"ucker coordinates $\Delta_J(N_p)$ for $J>\beta_p$ and the non-vanishing of $\Delta_{\beta_p}(N_p)$ implies that we find pivot elements in the rows $i\in\beta_p$ for each $p\in Q_0$ for a matrix presentation $w$ of $N$ in reduced echelon form. This shows that there is a matrix presentation $w$ of $N$ that satisfies \eqref{part1}--\eqref{part4}. Since $\cB_p\subset N_p$, the matrix $w$ is a block matrix and satisfies \eqref{part5}.
\end{proof}

\subsection{Defining equations}\label{subsection: defining equations}

Lemma \ref{lemma: normal form} identifies $C_\beta^M$ with a subset of the affine matrix space $\Mat_{\cB\times\cB}$. The following lemma determines defining equations (next to equations \eqref{part1}--\eqref{part5} from Section \ref{subsection: normal form of matrix representations}) for $C_\beta^M$, which shows that $C_\beta^M$ is a closed subvariety of $\Mat_{\cB\times\cB}$.

Let $T=\Gamma(M,\cB)$ be the coefficient quiver of $M$ w.r.t.\ $\cB$ and $F:T\to Q$ the canonical morphism. Recall that $T_0=\cB$.

\begin{lemma}\label{lemma: key formula}
 A matrix $w=(w_{i,j})_{i,j\in\cB}$ in $\beta$-normal form is the matrix representation of a point $N$ of $C_\beta^M$ if and only if $w$ satisfies for all arrows $\overline\alpha\in Q_1$ and all vertices $s\in F^{-1}(s(\overline \alpha))$ and $t\in F^{-1}(t(\overline \alpha))$ the equation
 \[         \tag*{{$E(\overline\alpha,t,s):$}} \label{key formula}
  \sum_{\substack{\alpha\in F^{-1}(\overline\alpha)\\\text{with }t(\alpha)=t}}   w_{s(\alpha),s}   \qquad = \qquad    \sum_{\substack{\alpha\in F^{-1}(\overline\alpha)}}   w_{t,t(\alpha)} w_{s(\alpha),s}.
 \]
 If $t\in\beta$ or $s\notin\beta$, then $E(\overline\alpha,t,s)$ is satisfied for any $w$ in $\beta$-normal form.
\end{lemma}

\begin{proof}
 Given a matrix $w=(w_{i,j})_{i,j\in\cB}$ in $\beta$-normal form, we write $w_i$ for the column vector $(w_{i,j})_{j\in\cB_p}$ where $p\in Q_0$ and $i\in\cB_p$. The matrix $w$ represents a point $N$ of $\Gr_\ue(M)$ if and only if for all $\overline\alpha\in Q_1$ and all $s\in\beta_{s(\overline\alpha)}$, there are $\lambda_k\in\C$ for $k\in \beta_{t(\overline\alpha)}$ such that
 \[
  M_{\overline\alpha} \, w_s \quad = \quad \sum_{k\in \beta_{t(\overline\alpha)}}\lambda_k w_k.
 \]
 This means that for all $t\in F^{-1}(t(\overline\alpha))$,
 \[
  \sum_{\substack{\alpha\in F^{-1}(\overline\alpha)\\\text{with }t(\alpha)=t}} w_{s(\alpha),s} \ = \ \bigl[ M_{\overline\alpha} \, w_s \bigr]_t \quad\text{equals}\quad \sum_{k\in\beta_{t(\overline\alpha)}} \lambda_k w_{t,k}.
 \]
 For $t\in\beta_{t(\overline\alpha)}$, we obtain that
 \[
    \sum_{\substack{\alpha\in F^{-1}(\overline\alpha)\\\text{with }t(\alpha)=t}} w_{s(\alpha),s} \quad = \quad \sum_{k\in\beta_{t(\overline\alpha)}} \lambda_k w_{t,k} \quad = \quad \sum_{k\in\beta_{t(\overline\alpha)}} \lambda_k \delta_{t,k} \quad = \quad \lambda_t
 \]
 by \eqref{part1} and \eqref{part2} for $w$ in $\beta$-normal form. Therefore, we obtain for arbitrary $t\in F^{-1}(t(\overline\alpha))$ that
 \[
  \sum_{\substack{\alpha\in F^{-1}(\overline\alpha)\\\text{with }t(\alpha)=t}} w_{s(\alpha),s} \quad = \quad \sum_{k\in\beta_{t(\overline\alpha)}} \Bigl( \sum_{\substack{\alpha\in F^{-1}(\overline\alpha)\\\text{with }t(\alpha)=k}} w_{s(\alpha),s}  \Bigr) w_{t,k} \quad = \quad \sum_{\substack{\alpha\in F^{-1}(\overline\alpha)}}   w_{t,t(\alpha)} w_{s(\alpha),s}
 \]
 as claimed. If $t\in\beta$, then this equation is satisfied for all $w$ in $\beta$-normal form by the definition of the $\lambda_k$ and since $w_{t,k}=\delta_{t,k}$ for $t\in\beta$. If $s\notin\beta$, then all coefficients $w_{s(\alpha),s}$ are $0$, i.e.\ we obtain the tautological equation $0=0$. This proves the latter claim of the lemma.
\end{proof}

\subsection{Relevant pairs and relevant triples}\label{subsection: relevant pairs and triples}

A \emph{relevant pair} is an element of the set
\[
 \Rel^2 \quad = \quad \{ \ (i,j) \in T_0\times T_0 \ | \ F(i)=F(j)\text{ and }i\leq j \ \}
\]
and an \emph{relevant triple} is an element of the set
\[
 \Rel^3 \quad = \quad \left\{ \ (\overline\alpha,t,s) \in Q_1\times T_0\times T_0 \ \left|  \begin{array}{c} \text{There is an }\alpha':s'\to t'\text{ in }T \text{ with }F(\alpha')=\overline\alpha, \\ F(s')=F(s),\, F(t')=F(t),\, s'\leq s\text{ and }t\leq t' \end{array} \right. \right\}.
\]

Given a matrix $w=(w_{i,j})$ in $\beta$-normal form, we say that $w_{i,j}$ is a \emph{constant coefficient (w.r.t.\ $\beta$)} if it appears in one of the equations \eqref{part1}--\eqref{part5} from Section \ref{subsection: normal form of matrix representations}, and otherwise we say that $w_{i,j}$ is a \emph{free coefficient (w.r.t.\ $\beta$)}, which is the case if and only if there is a $p\in Q_0$ such that $i\in\cB_p-\beta_p$, $j\in\beta_p$ and $i<j$. The significance of $\Rel^2$ is that if $w_{i,j}$ is not constant equal to $0$ w.r.t.\ $\beta$ (for any $\beta$), then $(i,j)$ is a relevant pair. 

If we substitute for a given $\beta$ all constant coefficients $w_{i,j}$ with $i\neq j$ by $0$, then we obtain \emph{$\beta$-reduced form of $E(\overline\alpha,t,s)$}:
\begin{equation}\label{eq: reduced key formula}
 \sum_{\substack{\alpha\in F^{-1}(\overline\alpha)\text{ with}\\ t(\alpha)=t,\ s(\alpha)\leq s,\\ s(\alpha)\notin\beta\text{ or }s(\alpha)=s}}   w_{s(\alpha),s}   \qquad = \qquad    
 \sum_{\substack{\alpha\in F^{-1}(\overline\alpha)\text{ with}\\ t< t(\alpha),\ s(\alpha)< s,\\ t(\alpha)\in\beta,\ s(\alpha)\notin\beta}}   w_{t,t(\alpha)} w_{s(\alpha),s} \quad + \quad \sum_{\substack{\alpha\in F^{-1}(\overline\alpha)\text{ with}\\ s(\alpha)=s,\ t\leq t(\alpha),\\ t(\alpha)\in\beta\text{ or }t(\alpha)=t}} w_{t,t(\alpha)},
\end{equation}

The significance of $\Rel^3$ is that if $E(\overline\alpha,t,s)$ is a non-trivial equation in the coefficients of a matrix $w$ in $\beta$-normal form (for any $\beta$), then $(\overline\alpha,t,s)$ is a relevant triple. 

In the following, we will associate certain values with relevant pairs and relevant triples. Since $T_0=\cB$ is linearly ordered, we can identify it order-preservative with $\{1,\dotsc,n\}$. We define the root of a connected component of $T$ as its smallest vertex, and we denote by $r(i)$ the root of the component that contains the vertex $i$. In particular, if $T$ is connected, then $1$ is the only root and $r(i)=1$ for all $i\in T_0$. Let $d(i,j)$ denote the graph distance of two vertices $i,j\in T_0$. We define the \emph{root distance of a relevant pair $(i,j)$} as 
\[
 \delta(i,j) \quad = \quad \max \, \bigr\{\, d\bigl(i,r(i)\bigr),\, d\bigl(j,r(j)\bigr) \, \bigr\}.
\]
We define the \emph{fibre length of a relevant pair $(i,j)$} as 
\[
 \epsilon(i,j) \quad = \quad \# \, \bigr\{ \ k \in T_0 \ \bigl| \ F(k)=F(i)\text{ and }i\leq k < j \ \bigr\}.
\]
We consider $\N\times\N\times T_0$ with its lexicographical order, i.e.\ $(i,j,k)<(i',j',k')$ if $i<i'$, or $i=i'$ and $j<j'$, or $i=i'$, $j=j'$ and $k<k'$. The inclusion 
\[
 \begin{array}{cccc}
    \Psi: & \Rel^2 & \longrightarrow & \N\times\N\times T_0 \\
          & (i,j) & \longmapsto      & \bigl(\epsilon(i,j),\delta(i,j),j\bigr) 
 \end{array}
\]
%$\Psi:\Rel^2\to \N\times\N\times T_0$ with $\Psi(i,j)=(\epsilon(i,j),\delta(i,j),j)$ 
induces a linear order on $\Rel^2$, i.e.\ $(i,j)<(i',j')$ if $\Psi(i,j)<\Psi(i',j')$.

Let $(\overline\alpha, t,s)$ be a relevant triple. We define $\Psi(\overline\alpha, t,s)$ as the maximum of $\Psi(s_{\min},s)$ and $\Psi(t,t_{\max})$ where $s_{\min}$ is the smallest vertex that is the source of an arrow $\alpha\in F^{-1}(\overline\alpha)$ with $t\leq t(\alpha)$ and $t_{\max}$ is the largest vertex that is the target of an arrow $\alpha\in F^{-1}(\overline\alpha)$ with $s(\alpha)\leq s$. 

For a relevant triple $(\overline\alpha, t,s)$ with $t\notin\beta$ and $s\in\beta$, we define $\Psi_\beta(\overline\alpha,t,s)$ as $\Psi(i,j)$ where $(i,j)$ is the largest relevant pair that appears as an index in the $\beta$-reduced form \eqref{eq: reduced key formula} of $E(\overline\alpha, t,s)$. Note that $E(\overline\alpha, t,s)$ contains at least one non-trivial term by the definition of a relevant triple. Note further that if there is an arrow $\alpha:s\to t$ in $F^{-1}(\overline\alpha)$ and every other arrow $\alpha'\in F^{-1}(\overline\alpha)$ satisfies either $s<s(\alpha')$ or $t(\alpha')<t$, then the only non-trivial terms in \eqref{eq: reduced key formula} are the constant coefficients $w_{s,s}$ and $w_{t,t}$. Thus in this case $\Psi_\beta(\overline\alpha,t,s)=\max\{\Psi(s,s),\Psi(t,t)\}$.

Since $w_{i,j}=0$ if $j<i$ for $w$ in $\beta$-normal form, we have $\Psi_\beta(\overline\alpha, t,s)\leq\Psi(\overline\alpha, t,s)$. In Section \ref{subsection: maximal relevant pairs}, we consider cases in which $\Psi_\beta(\overline\alpha, t,s)$ and $\Psi(\overline\alpha, t,s)$ are equal.

\begin{ex}
 A good example to illustrate the roles of relevant pairs, relevant triples and the function $\Psi$ is the following. Let $M$ be the preinjective representation of the Kronecker quiver $Q=K(2)$ with dimension vector $(3,4)$. Denote the two arrows of $Q$ by $\overline\alpha$ and $\overline\gamma$. Then there exists an ordered basis $\cB$ of $M$ such that the coefficient quiver $T=\Gamma(M,\cB)$ looks like
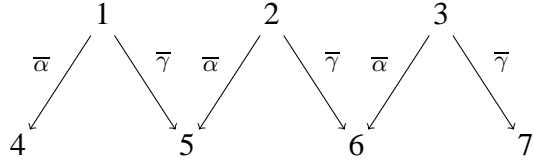
\[
   \begin{tikzpicture}[description/.style={fill=white,inner sep=0pt}]
  \matrix (m) [matrix of math nodes, row sep=3em, column sep=1.5em, text height=1ex, text depth=0ex]
   {   & 1 &   & 2 &   & 3     \\
     4 &   & 5 &   & 6 &   & 7 \\};
   \path[->,font=\scriptsize]
   (m-1-2) edge node[auto,swap] {$\overline\alpha$} (m-2-1)
   (m-1-2) edge node[auto] {$\overline\gamma$} (m-2-3)
   (m-1-4) edge node[auto,swap] {$\overline\alpha$} (m-2-3)
   (m-1-4) edge node[auto] {$\overline\gamma$} (m-2-5)
   (m-1-6) edge node[auto,swap] {$\overline\alpha$} (m-2-5)
   (m-1-6) edge node[auto] {$\overline\gamma$} (m-2-7);
  \end{tikzpicture}
\]
where we label the arrows by their image under $F$. We investigate the Schubert cell $C_\beta^M$ for $\beta=\{3,6,7\}$. A matrix $w=(w_{i,j})_{i,j\in\cB}$ in $\beta$-normal form has the six free coefficients $w_{1,3}$, $w_{2,3}$, $w_{4,6}$, $w_{5,6}$, $w_{4,7}$, $w_{5,7}$, and $w_{3,3}=w_{6,6}=w_{7,7}=1$. All other coefficients vanish. The non-trivial equations on the free coefficients are labeled by the relevant triples $(\overline\alpha,5,3)$, $(\overline\alpha,4,3)$, $(\overline\gamma,5,3)$ and $(\overline\gamma,4,3)$, and their respective $\beta$-reduced forms are
\[
 w_{2,3} \ = \ w_{5,6}, \qquad w_{1,3} \ = \ w_{4,6}, \qquad w_{1,3} \ = \ w_{2,3}w_{5,6} + w_{5,7} \quad \text{and} \quad 0 \ = \ w_{2,3} w_{4,6} + w_{4,7}.
\]
It is easy to see that these equations can be solved successively in linear terms. We show how these equations are organized by the ordering of $\Rel^2$ defined by $\Psi$. The relevant pairs that appear as indices of free coefficients are ordered as follows: 
\[
 (5,6) \ < \ (2,3) \ < \ (4,6) \ < \ (1,3) \ < \ (5,7) \ < \ (4,7).
\]
Ordered by size, we have
\[
 \Psi_\beta(\overline\alpha,5,3) = (2,3), \quad \Psi_\beta(\overline\alpha,4,3) = (1,3), \quad \Psi_\beta(\overline\gamma,5,3) = (5,7), \quad \Psi_\beta(\overline\gamma,4,3) = (4,7),
\]
which correspond to the indices of linear terms in each of the corresponding equations. Therefore, we find a unique solution in $w_{2,3}$, $w_{1,3}$, $w_{5,7}$ and $w_{4,7}$ for every $w_{5,6}$ and $w_{4,6}$, which shows that $C_\beta^M$ is isomorphic to $\A^2$.

This demonstrates how the ordering of relevant pairs organizes the defining equations for $C_\beta^M$ in a way that they are successively solvable in a linear term. In the following section, we will develop criteria under which this example generalizes to other representations $M$ and ordered bases $\cB$.
\end{ex}

\begin{rem}\label{rem: different orderings}
 The definition of $\Psi$ is based on heuristics with random examples of tree modules with ordered $F:T\to Q$. It is possible that different orders of $\Rel^2$  lead to analogues of Theorem \ref{thm: main theorem} that include quiver Grassmannians not covered in this text. Interesting variants might include the graph distance $d(i,j)$ of $i$ and $j$ as an ordering criterion; e.g.\ consider the ordering of $\Rel^2$ given by the map $\tilde\Psi:\Rel^2\to\N\times\N\to T_0$ with $\tilde\Psi(i,j)=(d(i,j),\epsilon(i,j),j)$. This might be of particular interest for exceptional modules that do not have an ordered tree basis such that $F:T\to Q$ is ordered. See, however, Section \ref{subsection: limiting examples} for some limiting examples.
\end{rem}

%%%%%%%%%%%%%%%%%%%%%%%%%%%%%%%%%%%%%%%%%%%%%%%%%%%%%%%%%%%%%%%%%%%%%%%%%%%%%%%%%%%%%%%%%%%%%%%%%%%%%%%%%%%%%%%%%%%%%%%%%%%%%%%%%%%%%%%%%%%%%%%%%%%%%%%%%%%%%%%%%%%%%
%%%%%%%%%%%%%%%%%%%%%%%%%%%%%%%%%%%%%%%%%%%%%%%%%%%%%%%%%%%%%%%%%%%%%%%%%%%%%%%%%%%%%%%%%%%%%%%%%%%%%%%%%%%%%%%%%%%%%%%%%%%%%%%%%%%%%%%%%%%%%%%%%%%%%%%%%%%%%%%%%%%%%

\section{Preliminaries for the main theorem}
\label{section: preliminaries}

\noindent
In this section, we develop the terminology and establish preliminary facts to formulate and prove the main theorem in Section \ref{section: schubert decompositions for tree modules}. As before, we let $Q$ be a quiver and $M$ a representation with ordered basis $\cB$ and dimension vector $\ud$. Let $\ue$ be another dimension vector for $Q$ and $\beta\subset\cB$ of type $\ue$. Let $T=\Gamma(M,\cB)$ be the coefficient quiver of $M$ w.r.t.\ $\cB$ and $F:T\to Q$ the canonical morphism. We identify the linearly ordered set $T_0=\cB$ with $\{1,\dotsc,n\}$.

\subsection{Extremal successor closed subsets}\label{subsection: extremal successor closed subsets}

An arrow $\alpha:s\to t$ in $T$ is called \emph{extremal (w.r.t.\ $F$)} if all other arrows $\alpha':s'\to t'$ with $F(\alpha')=F(\alpha)$ satisfy that either $s<s'$ or $t'<t$. Note that if $F$ is ordered and unramified, then every arrow of $T$ is extremal.

Recall that $T_0=\cB$, which allows us to consider $\beta$ as a subset of $T_0$. We say that $\beta$ is \emph{extremal successor closed} if for all extremal arrows $\alpha:s\to t$ of $T$, either $s\notin\beta$ or $t\in\beta$. Note that if $F$ is ordered and unramified, then $\beta$ is extremal successor closed if and only if $\beta$ is successor closed in the sense of \cite{Cerulli11} and \cite{Haupt12}. 

\begin{lemma}\label{lemma: not empty implies extremal successor closed}
 If $\beta$ is not extremal successor closed, then $C_\beta^M$ is empty.
\end{lemma}

\begin{proof}
 We assume that $C_\beta^M$ is non-empty and prove the lemma by contraposition. Let $\alpha:s\to t$ be an extremal arrow in $T$ and $\overline\alpha=F(\alpha)$. Let $N\in C_\beta^M$ have the matrix representation $w$ in $\beta$-normal form. The $\beta$-reduced form of $E(\overline\alpha,t,s)$ is
 \[
  w_{s,s} \quad = \quad w_{t,t} w_{s,s}
 \]
 since $\alpha:s\to t$ is extremal and thus for every other $\alpha':s'\to t'$ in $F^{-1}(\overline\alpha)$ either $s<s'$ and thus $w_{s',s}=0$ or $t'<t$ and thus $w_{t',t}=0$. Since $w_{s,s}=1$ if $s\in\beta$ (according to \eqref{part1}) and $w_{t,t}=0$ if $t\notin\beta$ (according to \eqref{part4}), equation $E(\overline\alpha,t,s)$ would be $1=0$ if $s\in\beta$ and $t\notin\beta$. This is not possible since we assumed that $C_\beta$ is non-empty. Therefore $s\notin\beta$ or $t\in\beta$, which shows that $\beta$ is extremal successor closed.
\end{proof}

\subsection{Ordered and ramified morphisms}\label{subsection: ordered morphisms}

The morphism $F:T\to Q$ is \emph{ordered} if for all arrows $\alpha:s\to t$ and $\alpha':s'\to t'$ of $T$ with $F(\alpha)=F(\alpha')$, we have $s\leq s'$ if and only if $t\leq t'$.

Consider an arrow $\overline\alpha\in Q_1$ and a vertex $i\in T_0$. The \emph{ramification index $r_{\overline\alpha}(i)$ at $i$ in direction $\overline\alpha$} is the number of arrows $\alpha\in F^{-1}(\overline\alpha)$ with source or target $i$. If $r_{\overline\alpha}(i)>1$, we say that \emph{$F$ branches at $i$ in direction $\overline\alpha$} and that \emph{$F$ ramifies above $F(i)$}. The morphism $F:T\to Q$ is \emph{unramified} or a \emph{winding} if for all $\overline\alpha\in Q_1$ and all $i\in T_0$, we have $r_{\overline\alpha}(i)\leq 1$. In other words, $F:T\to Q$ is unramified if and only if the associated map of CW-complexes is unramified. 

Note that $F$ is strictly ordered (in the sense of \cite[Section 4.2]{L12}) if and only if $F$ is ordered and unramified. From this viewpoint, we can say that we extend Theorem 4.2 of \cite{L12} from unramified morphisms $F:T\to Q$ to ramified $F$ in this text.

\subsection{Polarizations}\label{subsection: polarizations}

Let $I=\{i_1,\dotsc,i_r\}$ be a finite ordered set with $i_1<\dotsb<i_r$. A \emph{sorting of $I$} is a decomposition $I=I^<\amalg I^>$ such that $I^<=\{i_1,\dotsc,i_s\}$ and $I^>=\{i_{s+1},\dotsc,i_r\}$ for some $s\in\{1,\dotsc,r-1\}$. A \emph{polarization for a linear map $M_{\overline\alpha}:M_p\to M_q$} (between finite dimensional complex vector spaces) are ordered bases $\cB_p$ and $\cB_q$ for $M_p$ and $M_q$, respectively, that admit sortings $\cB_p=\cB_{p,\overline\alpha}^<\amalg\cB_{p,\overline\alpha}^>$ and $\cB_q=\cB_{q,\overline\alpha}^<\amalg\cB_{q,\overline\alpha}^>$ such that $M_{\overline\alpha}$ restricts to a surjection $\cB_{p,\overline\alpha}^<\cup\{0\}\twoheadrightarrow \cB_{q,\overline\alpha}^<\cup\{0\}$ and its adjoint map $M_{\overline\alpha}^\ad$ restricts to a surjection $\cB_{q,\overline\alpha}^>\cup\{0\}\twoheadrightarrow\cB_{p,\overline\alpha}^>\cup\{0\}$. We call these decompositions of $\cB_p$ and $\cB_q$ a \emph{sorting for $M_{\overline\alpha}$}.

Let $M$ be a representation of $Q$. A \emph{polarization for $M$} is an ordered basis $\cB$ of $M$ such that $\cB_p$ and $\cB_q$ are a polarization for every arrow $\overline\alpha:p\to q$ in $Q$. In this case, we also say that \emph{$M$ is polarized by $\cB$}. An \emph{ordered polarization of $M$} is a polarization $\cB$ such that the canonical morphism $F:T\to Q$ from the coefficient quiver is ordered.

In other words, $M$ is polarized by $\cB$ if and only if there are for all arrows $\overline\alpha:p\to q$ in $Q$ sortings $\cB_p=\cB_{p,\overline\alpha}^<\amalg\cB_{p,\overline\alpha}^>$ and $\cB_q=\cB_{q,\overline\alpha}^<\amalg\cB_{q,\overline\alpha}^>$ such that $r_{\overline\alpha}(i)\leq 1$ for all $i\in \cB_{p,\overline\alpha}^<\amalg\cB_{q,\overline\alpha}^>$ and $r_{\overline\alpha}(i)\geq 1$ for all $i\in \cB_{q,\overline\alpha}^<\amalg\cB_{p,\overline\alpha}^>$. This means that the non-zero matrix coefficients of $M_{\overline\alpha}$ w.r.t.\ $\cB_p$ and $\cB_q$ can be covered by an upper left submatrix $M_{\overline\alpha}^<$ and a lower right submatrix $M_{\overline\alpha}^>$ where $M_{\overline\alpha}^<$ has at most one non-zero entry in each column and at least one non-zero entry in each row while $M_{\overline\alpha}^>$ has at least one non-zero entry in each column and at most one non-zero entry in each row.

The following figure illustrates the typical shape of a fibre of an arrow $\overline\alpha:p\to q$ of $Q$ in the coefficient quiver $T=\Gamma(M,\cB)$ where $\cB$ is an ordered polarization for $M$. We use the convention that we order the vertices from left to right in growing order. The property that $\cB$ is a polarization is visible by the number of arrows connecting to a vertex in the upper left / lower left / upper right / lower right of the picture, and the property that $F:T\to Q$ is ordered is visible from the fact that the arrows do not cross each other.
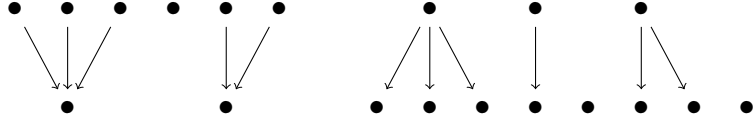
\[
   \begin{tikzpicture}[description/.style={fill=white,inner sep=0pt}]
  \matrix (m) [matrix of math nodes, row sep=2em, column sep=0.5em, text height=1ex, text depth=0ex]
   { \bullet & \bullet & \bullet & \bullet & \bullet & \bullet & \   &         & \bullet &         & \bullet &         & \bullet &         &         \\
             & \bullet &         &         & \bullet &         & \   & \bullet & \bullet & \bullet & \bullet & \bullet & \bullet & \bullet & \bullet \\};
   \path[->,font=\scriptsize]
   (m-1-1) edge (m-2-2)
   (m-1-2) edge (m-2-2)
   (m-1-3) edge (m-2-2)
   (m-1-5) edge (m-2-5)
   (m-1-6) edge (m-2-5)
   (m-1-9) edge (m-2-8)
   (m-1-9) edge (m-2-9)
   (m-1-9) edge (m-2-10)
   (m-1-11) edge (m-2-11)
   (m-1-13) edge (m-2-13)
   (m-1-13) edge (m-2-14);
  \end{tikzpicture}
\]

\begin{lemma}\label{lemma: existence of extremal arrows for polarized representations}
 Let $\cB$ be an ordered polarization for $M$. Let $\overline\alpha:p\to q$ be an arrow in $Q$ and  $\cB_p=\cB_{p,\overline\alpha}^<\amalg\cB_{p,\overline\alpha}^>$ and $\cB_q=\cB_{q,\overline\alpha}^<\amalg\cB_{q,\overline\alpha}^>$ a sorting for $M_{\overline\alpha}$. Then every $i$ in $\cB_{q,\overline\alpha}^<\amalg\cB_{p,\overline\alpha}^>$ connects to a unique extremal arrow.
\end{lemma}

\begin{proof}
 It is clear that every vertex $i$ connects at most to one extremal arrow in $F^{-1}(\overline\alpha)$. Since $M$ is polarized by $\cB$, we have that if $r_\alpha(i)\geq 1$, then $r_\alpha(j)= 1$ for all $j$ such that there is an arrow $\alpha:i\to j$ or $\alpha:j\to i$ in $F^{-1}(\overline\alpha)$. In case $i=s(\alpha)$, this means that $\alpha:i\to j_0$ is extremal where $j_0$ is minimal among the targets of arrows in $F^{-1}(\overline\alpha)$ with source $i$. In case $i=t(\alpha)$, this means that $\alpha:j_0\to i$ is extremal where $j_0$ is maximal among the sources of arrows in $F^{-1}(\overline\alpha)$ with target $i$. This establishes the lemma.
\end{proof}
  
\begin{rem}
 Ringel develops in \cite{Ringel12} the notion of a radiation basis in order to exhibit distinguished tree bases for exceptional modules. By Proposition 3 of \cite{Ringel12}, a radiation basis $\cB$ is a polarization of $M$ (w.r.t.\ any ordering of $\cB$). Examples of representations with radiation basis are indecomposable representations of Dynkin quivers (with an exception for $E_8$) and the pull-back of preinjective or preprojective modules of the Kronecker quiver $K(n)$ with $n$ arrows to its universal covering graph. Since the coefficient quiver of a pull-back is the same as the coefficient quiver of the original representation, it follows that every preinjective or preprojective representation of the Kronecker quiver $K(n)$ is polarized by some ordered basis.

 In Appendix \ref{appendix}, we find a general strategy to establish polarizations of exceptional modules along Schofield induction. In the joint forth-coming paper \cite{Lorscheid-Weist} with Thorsten Weist, we will show that every exceptional representation $M$ of a quiver of affine Dynkin type $\tilde D_n$ has a polarization which yields a Schubert decomposition of $\Gr_\ue(M)$ into affine spaces.
\end{rem}

\subsection{Maximal relevant pairs}\label{subsection: maximal relevant pairs}

Let $\overline\alpha\in Q_1$. A relevant pair $(i,j)$ is \emph{maximal for $\overline\alpha$} if there exists a relevant triple $(\overline\alpha,t,s)$ such that $\Psi(i,j)=\Psi(\overline\alpha,t,s)$. 

\begin{lemma}\label{lemma: maximal relevant pairs for relevant triples}
 Assume that $M$ is polarized by $\cB$ and that $\beta\subset\cB$ is extremal successor closed. Let $(\overline\alpha,t,s)$ be a relevant triple with $s\in\beta$ and $t\notin\beta$. Then one of the following holds true.
 \begin{enumerate}
  \item There is an extremal arrow $\alpha':s'\to t$ in $F^{-1}(\overline \alpha)$ such that $s'\notin\beta$ and
        \[
         \Psi_\beta(\overline\alpha,t,s) \ = \ \Psi(s',s) \ = \ \Psi(\overline\alpha,t,s).
        \]
        In this case, the $\beta$-reduced form of $E(\overline\alpha,t,s)$ is
        \[
         w_{s',s} \quad = \quad - \sum_{\substack{\alpha\in F^{-1}(\overline\alpha)\text{ with}\\ t(\alpha)=t,\ s(\alpha)\notin\beta}}  w_{s(\alpha),s} \ + \sum_{\substack{\alpha\in F^{-1}(\overline\alpha)\text{ with}\\ s'<s(\alpha)<s, \\ s(\alpha)\notin\beta,\ t(\alpha)\in\beta}} w_{t,t(\alpha)}w_{s(\alpha),s} \ + \sum_{\substack{\alpha\in F^{-1}(\overline\alpha)\\ \text{with }s(\alpha)=s,\text{ and}\\ t(\alpha)\in\beta\text{ or }t(\alpha)=t}} w_{t,t(\alpha)}.
        \]
  \item There is an extremal arrow $\alpha':s\to t'$ in $F^{-1}(\overline \alpha)$ such that $t'\in\beta$ and
        \[
         \Psi_\beta(\overline\alpha,t,s) \ = \ \Psi(t,t') \ = \ \Psi(\overline\alpha,t,s).
        \]
        In this case, the $\beta$-reduced form of $E(\overline\alpha,t,s)$ is
        \[
         w_{t,t'} \quad = \quad \sum_{\substack{\alpha\in F^{-1}(\overline\alpha)\\ \text{with }t(\alpha)=t,\text{ and}\\ s(\alpha)\notin\beta\text{ or }s(\alpha)=s}}  w_{s(\alpha),s} \ - \sum_{\substack{\alpha\in F^{-1}(\overline\alpha)\text{ with}\\ t<t(\alpha)<t',\\ s(\alpha)\notin\beta,\ t(\alpha)\in\beta}} w_{t,t(\alpha)}w_{s(\alpha),s} \ - \sum_{\substack{\alpha\in F^{-1}(\overline\alpha)\text{ with}\\ s(\alpha)=s,\ t(\alpha)\in\beta}} w_{t,t(\alpha)}.
        \]
 \end{enumerate}
\end{lemma}

\begin{proof}
 Once we know that there is an extremal arrow $\alpha':s'\to t$ (or $\alpha':s\to t'$), it is clear that $s'\notin\beta$ (or $t\in\beta$), that $w_{s',s}$ (or $w_{t,t'}$) is a free coefficient and that the $\beta$-reduced form of $E(\overline\alpha,t,s)$ looks as described in \eqref{part1} (or \eqref{part2}).

 If there are extremal arrows $\alpha':s'\to t$ and $\alpha'':s\to t''$, then $s'$ is minimal among the sources of arrows in $F^{-1}(\overline\alpha)$ with target $t$, and $t''$ is maximal among the targets of arrows in $F^{-1}(\overline\alpha)$ with source $s$. Clearly, we have $\Psi(\overline\alpha,t,s)=\max\{\Psi(s',s),\Psi(t,t'')\}$. By the definition of a relevant triple, we have $s'\leq s$ and $t\leq t''$. Since $\beta$ is extremal successor closed, $s'\notin\beta$ and $t''\in\beta$. In particular, this means that $s'\neq s$ and $t\neq t''$, and thus $w_{s',s}$ and $w_{t,t''}$ are free coefficients. By the minimality of $s'$ and the maximality of $t''$, every other free coefficient $w_{i,j}$ in the $\beta$-reduced form of $E(\overline\alpha,t,s)$ must satisfy $\epsilon(i,j)<\max\{\epsilon(s',s),\epsilon(t,t'')\}$. Therefore also $\Psi_\beta(\overline\alpha,t,s)=\max\{\Psi(s',s),\Psi(t,t'')\}$, which establishes the proposition in the case that both $s$ and $t$ connect to extremal arrows in the fibre of 
$\overline\alpha$.

 Let $p=s(\overline\alpha)$ and $q=t(\overline\alpha)$. Let $\cB_p=\cB_{p,\overline\alpha}^<\amalg\cB_{p,\overline\alpha}^>$ and $\cB_q=\cB_{q,\overline\alpha}^<\amalg\cB_{q,\overline\alpha}^>$ be sortings for $M_{\overline\alpha}$. %We assume that [...] is minimal among the targets of arrows in $F^{-1}(\overline\alpha)$ with source $i$. In case $i=t(\alpha)$, this means that $\alpha:j_0\to i$ is extremal where $j_0$ is maximal among the sources of arrows in $F^{-1}(\overline\alpha)$ with target 
 If $s$ is not the source of any extremal arrow in the fibre of $\overline\alpha$, then Lemma \ref{lemma: existence of extremal arrows for polarized representations} implies that $s\in\cB_{p,\overline\alpha}^<$. By the definition of a relevant triple, there is an arrow $\alpha\in F^{-1}(\overline\alpha)$ with $s(\alpha)\leq s$ and $t(\alpha)\leq t$. This implies that $t\in\cB_{q,\overline\alpha}^<$ and, by Lemma \ref{lemma: existence of extremal arrows for polarized representations}, that there is an extremal arrow $\alpha':s'\to t$. Since $\beta$ is extremely successor closed, $s'\notin\beta$ and $w_{s,s'}$ is a free coefficient.

 We claim that in this situation $\Psi_\beta(\overline\alpha,t,s)=\Psi(s',s)=\Psi(\overline\alpha,t,s)$. Since $\alpha'$ is extremal, all $s''\in F^{-1}(p)$ appearing in an index of the $\beta$-reduced form of $E(\overline\alpha,t,s)$ must lie between $s'$ and $s$. This means that $\epsilon(s',s)$ is larger than $\epsilon(s',s'')$ and $\epsilon(s'',s)$ if $s''$ is different from both $s$ and $s'$. Similarly, the largest relevant pair $(t'',t')$ with $F(t'')=F(t')=q$ satisfies $t''=t$ and that $t'$ is maximal among the targets of arrows in $F^{-1}(\overline\alpha)$ whose source is less or equal to $s$. Since $t,t'\in\cB_{q,\overline\alpha}^<$, we have $\epsilon(t,t')\leq\epsilon(s',s)$. Equality can only hold if every $s''$ between $s'$ and $s$ is the source of precisely one arrow in $F^{-1}(\overline\alpha)$. But then there would be such a unique arrow with source $s$, which is necessarily extremal. Since this contradicts the assumption that there is no extremal arrow with source $s$ in $F^{-1}(\overline\
alpha)$, we see that $\Psi(s',s)>\Psi(t,t')$. This shows that $\Psi_\beta(\overline\alpha,t,s)=\Psi(s',s)=\Psi(\overline\alpha,t,s)=\Psi(\overline\alpha,t,s)$, which means that \eqref{part1} is satisfied.

 If $t$ is not the target of any extremal arrow in the fibre of $\overline\alpha$, then we conclude analogously to the previous case that there is an extremal arrow $\alpha':s\to t'$ with $t'\in\beta$ such that $\Psi_\beta(\overline\alpha,t,s)=\Psi(t,t')$. Thus in this case, \eqref{part2} is satisfied.
\end{proof}

\begin{lemma}\label{lemma: a pair is maximal for at most one triple with fixed arrow}
 Let $\overline\alpha\in Q_1$ and $(i,j)\in\Rel^2$. Assume that $M$ is polarized by $\cB$.
 \begin{enumerate}
  \item If $F(i)=s(\overline\alpha)$, then there is at most one $\alpha:i\to t$ in $F^{-1}(\overline\alpha)$ such that $\Psi(i,j)=\Psi(\overline\alpha,t,j)$.
  \item If $F(j)=t(\overline\alpha)$, then there is at most one $\alpha:s\to j$ in $F^{-1}(\overline\alpha)$ such that $\Psi(i,j)=\Psi(\overline\alpha,i,s)$.
 \end{enumerate}
\end{lemma}

\begin{proof}
 We prove \eqref{part1}. If there is only one arrow $\alpha$ in $F^{-1}(\overline\alpha)$ with source $i$, then \eqref{part1} is clear. Assume that there are two different arrows $\alpha:i\to t$ and $\alpha':i\to t'$ in  $F^{-1}(\overline\alpha)$ with $t'<t$. Since $M$ is polarized, we have $r_{\overline\alpha}(k)\leq 1$ for all $k\geq i$ and $r_{\overline\alpha}(l)\geq 1$ for all $l\geq t'$. This means that there is an arrow $\alpha'':j\to t''$ and that $\epsilon(t',t'')>\epsilon(t,t'')\geq\epsilon(i,j)$. An equality $\epsilon(t,t'')=\epsilon(i,j)$ is only possible if $t$ is maximal among the targets of arrows in $F^{-1}(\overline\alpha)$ with source $i$. 

 This shows \eqref{part1}. The proof of \eqref{part2} is analogous.
\end{proof}

\begin{cor}\label{cor: a unique relevant triple per maximal relevant pair}
  Assume that $M$ is polarized by $\cB$ and that $\beta\subset\cB$ is extremal successor closed. Let $\overline\alpha\in Q_1$. If $(i,j)$ is maximal for $\overline\alpha$, then there is a unique $(\overline\alpha,t,s)\in\Rel^3$ such that $\Psi_\beta(\overline\alpha,t,s)=\Psi(i,j)=\Psi(\overline\alpha,t,s)$. If $(i,j)$ is not maximal for $\overline\alpha$, then there is no relevant triple $(\overline\alpha,t,s)$ with $\Psi(i,j)=\Psi(\overline\alpha,t,s)$.
\end{cor}

\begin{proof}
 This is an immediate consequence of Lemmas \ref{lemma: maximal relevant pairs for relevant triples} and \ref{lemma: a pair is maximal for at most one triple with fixed arrow}.
\end{proof}

%%%%%%%%%%%%%%%%%%%%%%%%%%%%%%%%%%%%%%%%%%%%%%%%%%%%%%%%%%%%%%%%%%%%%%%%%%%%%%%%%%%%%%%%%%%%%%%%%%%%%%%%%%%%%%%%%%%%%%%%%%%%%%%%%%%%%%%%%%%%%%%%%%%%%%%%%%%%%%%%%%%%%%%%%%%%%%%%%%%%%%%%%%%%%%%%%%%%%%%%%%%%%%%%%%%%%%%%%%%%

\section{Schubert decompositions for tree modules}
\label{section: schubert decompositions for tree modules}

%\subsection{Extremal successor closed subsets}\label{subsection: extremal successor closed subsets}
%\subsection{The main theorem}\label{subsection: main theorem}

\begin{thm} \label{thm: main theorem}
 Let $M$ be a representation of $Q$ and $\cB$ an ordered polarization for $M$. Let $\ue$ be a dimension vector for $Q$. Assume that every $(i,j)\in\Rel^2$ is maximal for at most one $\overline\alpha\in Q_1$. Then
 \[
  \Gr_\ue(M) \quad = \quad \coprod_{\substack{\beta\subset\cB\\ \text{of type }\ue}} C_\beta^M
 \]
 is a decomposition into affine spaces. Moreover, $C_\beta^M$ is not empty if and only if $\beta$ is extremal successor closed.
\end{thm}

\begin{proof}
 By Lemma \ref{lemma: not empty implies extremal successor closed}, $C_\beta^M$ is empty if $\beta$ is not extremal successor closed. Let $\beta$ be extremal successor closed. The theorem is proven once we have shown that $C_\beta^M$ is an affine space,

 As before, we identify $T_0$ order-preservative with $\{1,\dotsc,n\}$. For $\psi\in\N\times\N\times T_0$, we denote by $C_\beta^M(\psi)$ the solution space of all coefficients $w_{i,j}$ with $\Psi(i,j)\leq\psi$ in all equations $E(\overline\alpha,t,s)$ where $(\overline\alpha,t,s)$ is a relevant triple with $\Psi_\beta(\overline\alpha,t,s)<\psi$. We show by induction over $\psi\in\Psi(\Rel^2)$ that $C_\beta^M(\psi)$ is an affine space. Since $\Psi(\Rel^2)$ is finite, this implies that $C_\beta^M$ is an affine space as required.

 As base case, consider $\psi=\Psi(n,n)$. By Lemma \ref{lemma: key formula}, only those relevant triples $(\overline\alpha,t,s)$ with $t\notin\beta$ and $s\in\beta$ lead to non-trivial equations $E(\overline\alpha,t,s)$. For such a relevant triple, $\Psi_\beta(\overline\alpha,t,s)\leq\psi$ if and only if $E(\overline\alpha,t,s)$ does not contain any free coefficient and thus is of the form $w_{s,s}=w_{t,t}w_{s,s}$. This is the case if and only if there is an extremal arrow $\alpha:s\to t$ in $F^{-1}(\overline\alpha)$. Since $\beta$ is extremal successor closed, $w_{s,s}=w_{t,t}w_{s,s}$ is satisfied. This means that $C_\beta^M(\psi)=\A^0$ is a point.

 Consider $\psi>\Psi(n,n)$ and let $\psi'$ be its predecessor in $\Psi(\Rel^2)$. We assume that $C_\beta^M(\psi')$ is an affine space. By the assumption of the theorem, $(i,j)$ is maximal for at most one $\overline\alpha\in Q_1$. If there is none such $\overline\alpha$, then there is no relevant triple $(\overline\alpha,t,s)$ with $\Psi_\beta(\overline\alpha,t,s)=\Psi(i,j)$, which means that $w_{i,j}$ does not appear as a maximal coefficient of an equation $E(\overline\alpha,t,s)$. If $i\in\beta$ or $j\notin\beta$, then $w_{i,j}=0$ and $C_\beta^M(\psi)=C_\beta^M(\psi')$. Otherwise $w_{i,j}$ is free and $C_\beta^M(\psi)=C_\beta^M(\psi')\times\A^1$.

 If there is an arrow $\overline\alpha\in Q_1$ such that $(i,j)$ is maximal for $\overline\alpha$, then there exists a unique relevant triple $(\overline\alpha,t,s)$ such that $\Psi_\beta(\overline\alpha,t,s)=\Psi(i,j)$ by Corollary \ref{cor: a unique relevant triple per maximal relevant pair}. If $w_{i,j}$ is not free, then $i\in\beta$ or $j\notin\beta$. By Lemma \ref{lemma: maximal relevant pairs for relevant triples}, either $t=i$ and there is an extremal arrow $\alpha:s\to j$ in $F^{-1}(\overline\alpha)$ or $s=j$ and there is an extremal arrow $\alpha:i\to t$ in $F^{-1}(\overline\alpha)$. In either case, if $i\in\beta$ or $j\notin\beta$, then $t\in\beta$ or $s\notin\beta$ since $\beta$ is extremal successor closed. This means that $E(\overline\alpha,t,s)$ is trivial and thus $C_\beta^M(\psi)=C_\beta^M(\psi')$. If $w_{i,j}$ is free, but $E(\overline\alpha,t,s)$ is trivial, then $C_\beta^M(\psi)=C_\beta^M(\psi')\times\A^1$. If finally $w_{i,j}$ is free and $E(\overline\alpha,t,s)$ is non-trivial, then $w_{
i,j}$ is determined by all coefficients $w_{i',j'}$ with $\Psi(i',j')<\Psi(i,j)$ by one of the formulas in Lemma \ref{lemma: maximal relevant pairs for relevant triples}. This means that $C_\beta^M(\psi)=C_\beta^M(\psi')$.

 Thus we have shown that in all possible cases, $C_\beta^M(\psi)$ equals either $C_\beta^M(\psi')$ or $C_\beta^M(\psi')\times\A^1$, which are both affine spaces by the inductive hypothesis. This finishes the proof of the theorem.
\end{proof}

\begin{rem}
 Though the assumptions of Theorem \ref{thm: main theorem} come in a different shape than the Hypothesis (H) in Section 4.5 of \cite{L12}, they are indeed equivalent to Hypothesis (H) if $F:Q\to T$ is unramified. 
\end{rem}

\begin{rem}
 Though we do not explicitly require that $\cB$ is a tree basis, it follows from the other assumptions of the theorem that $M$ is a tree module. Indeed, if the coefficient quiver $T$ had a loop and $i$ was the largest vertex of this loop in maximal distance to $1$, then the relevant pair $(i,i)$ would be maximal for the two connecting arrows of the loop. Note that if $M$ is not indecomposable, then $T=\Gamma(M,\cB)$ is not necessarily connected (cf.\ Example \ref{ex: del pezzo}). 

 By \cite{Ringel98}, every exceptional module is a tree module. But it is clear that not every exceptional module admits an ordered tree basis such that the canonical morphism $F:T\to Q$ from the coefficient quiver is ordered. For instance, there are exceptional representations of the Kronecker quiver $K(3)$ with three arrows that attest to this fact, cf.\ the example $P(x,3)$ in \cite[p.\ 15]{Ringel12}.

 However, if $M$ has an radiation basis $\cB$, then we can order $\cB$ inductively along the construction of $M$ by smaller radiation modules such that $\cB$ satisfies the assumptions of the theorem. In particular, this includes all exceptional representations of Dynkin type, with an exception for $E_8$. We see that the class of modules that admit an ordered basis to that we can apply the theorem lies somewhere between radiation modules and tree modules.
\end{rem}

\begin{cor}\label{cor: formula for the euler characteristic}
 Under the assumptions of Theorem \ref{thm: main theorem}, the Euler characteristic of $\Gr_\ue(M)$ equals the number of extremal successor closed subsets $\beta\subset\cB$ of type $\ue$.
\end{cor}

\begin{proof}
 Since the Euler characteristic is additive under decompositions into locally closed subsets, 
 \[
  \chi\bigl(\;\Gr_\ue(M)\;\bigr) \quad = \quad \sum_{\substack{\beta\subset\cB\\ \text{of type }\ue}} \chi\bigl(\;C_\beta^M\;\bigr).
 \]
 The Euler characteristic of an affine space is $1$ and the Euler characteristic of the empty set is $0$. Therefore the corollary follows immediately from Theorem \ref{thm: main theorem}.
\end{proof}

\begin{cor}\label{cor: schubert cells form a basis of the cohomology ring}
 If $\Gr_\ue(M)$ is smooth and the assumptions of Theorem \ref{thm: main theorem} are satisfied, then the closures of the non-empty Schubert cells $C_\beta^M$ of $\Gr_\ue(M)$ represent an additive basis for the cohomology ring $H^*(\Gr_\ue(M))$. If $n=\dim\Gr_\ue(M)$ and $d=\dim C_\beta^M$, then the class of the closure of $C_\beta^M$ is in $H^{n-2d}(\Gr_\ue(M))$.
\end{cor}

\begin{proof}
 This follows immediately from \cite[Cor.\ 6.2]{L12}.
\end{proof}

\subsection{Two examples for type $D_4$}

\begin{ex}[A quiver Grassmannian of a ramified tree module]\label{ex: ramified d_4}
 The following example is an instance of a ramified tree module to which the methods of this text apply. Let $Q$ be the quiver 
 
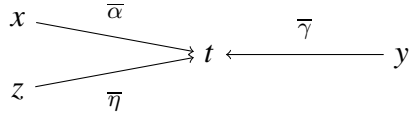
\[
  \begin{tikzpicture}[description/.style={fill=white,inner sep=0pt}]
   \matrix (m) [matrix of math nodes, row sep=0em, column sep=5em, text height=1ex, text depth=0ex]
    { x &   & \\
        & t & y \\
      z &   & \\};
    \path[->,font=\scriptsize]
    (m-1-1) edge node[above=3pt] {$\overline\alpha$} (m-2-2)
    (m-3-1) edge node[below=3pt] {$\overline\eta$} (m-2-2)
    (m-2-3) edge node[above=2pt] {$\overline\gamma$} (m-2-2);
  \end{tikzpicture}
 \]
 of type $D_4$ and let $M$ be the exceptional module
 
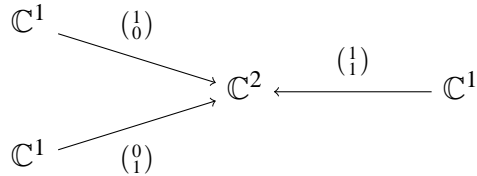
\[
  \begin{tikzpicture}[description/.style={fill=white,inner sep=0pt}]
   \matrix (m) [matrix of math nodes, row sep=1em, column sep=5em, text height=1ex, text depth=0ex]
    { \C^1 &      & \\
           & \C^2 & \C^1 \\
      \C^1 &      & \\};
    \path[->,font=\scriptsize]
    (m-1-1) edge node[above=3pt] {$\left(\begin{smallmatrix} 1\\0\end{smallmatrix}\right)$} (m-2-2)
    (m-3-1) edge node[below=3pt] {$\left(\begin{smallmatrix} 0\\1\end{smallmatrix}\right)$} (m-2-2)
    (m-2-3) edge node[above=2pt] {$\left(\begin{smallmatrix} 1\\1\end{smallmatrix}\right)$} (m-2-2);
  \end{tikzpicture}
 \]
 of $Q$. We can order the obvious basis $\cB$ such that the coefficient quiver $T$ looks like
 
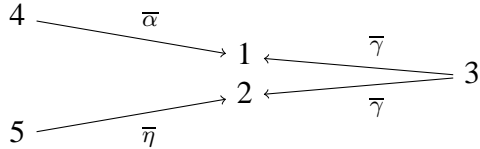
\[
  \begin{tikzpicture}[description/.style={fill=white,inner sep=0pt}]
   \matrix (m) [matrix of math nodes, row sep=-0.5em, column sep=1.5em, text height=1ex, text depth=0ex]
    { 4 &   &   &   &   &   &   &   &   \\
      \ &   &   &   &   &   &   &   &   \\     
        &   &   &   & 1 &   &   &   &    \\     
        &   &   &   &   &   &   &   & 3 \\     
        &   &   &   & 2 &   &   &   &    \\     
      \ &   &   &   &   &   &   &   &   \\     
      5 &   &   &   &   &   &   &   &   \\ };
    \path[->,font=\scriptsize]
    (m-1-1) edge node[auto] {$\overline\alpha$} (m-3-5)
    (m-4-9) edge node[auto,swap] {$\overline\gamma$} (m-3-5)
    (m-4-9) edge node[auto] {$\overline\gamma$} (m-5-5)
    (m-7-1) edge node[auto,swap] {$\overline\eta$} (m-5-5);
  \end{tikzpicture}
 \]
 where we label the arrows by its image under $F$. For the dimension vector $\ue$ with $e_x=e_z=0$ and $e_y=e_t=1$, we obtain precisely one subrepresentation $N$ of $M$ with $\udim N=\ue$. This means that $\Gr_\ue(M)$ is a point. Therefore, the Euler characteristic of $\Gr_\ue(M)$ equals $1$.

 There is precisely one extremal successor closed subset of type $\ue$, namely $\beta=\{2,3\}$, which accounts for the Euler characteristic. It is indeed easily verified that the assumptions of Theorem \ref{thm: main theorem} are satisfied. Note that $\beta$ is not successor closed, which shows that the number of successor closed subsets does not coincide with the Euler characteristic in this example.
\end{ex}

\begin{ex}[A del Pezzo surface of degree $6$]\label{ex: del pezzo}
 The previous representation appears as a subrepresentation of the following unramified representation. This example arose from discussions with Markus Reineke. Let $Q$ be the same quiver as in the previous example and $M$ the representation
 
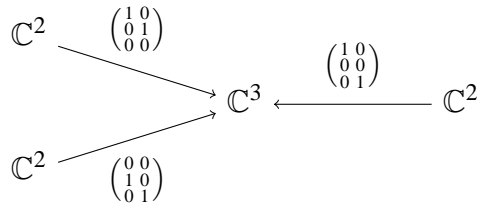
\[
  \begin{tikzpicture}[description/.style={fill=white,inner sep=0pt}]
   \matrix (m) [matrix of math nodes, row sep=1em, column sep=5em, text height=1ex, text depth=0ex]
    { \C^2 &      & \\
           & \C^3 & \C^2 \\
      \C^2 &      & \\};
    \path[->,font=\scriptsize]
    (m-1-1) edge node[above=3pt] {$\left(\begin{smallmatrix} 1&0\\0&1\\0&0\end{smallmatrix}\right)$} (m-2-2)
    (m-3-1) edge node[below=3pt] {$\left(\begin{smallmatrix} 0&0\\1&0\\0&1\end{smallmatrix}\right)$} (m-2-2)
    (m-2-3) edge node[above=2pt] {$\left(\begin{smallmatrix} 1&0\\0&0\\0&1\end{smallmatrix}\right)$} (m-2-2);
  \end{tikzpicture}
 \]
of $Q$. We can order the obvious basis $\cB$ such that the coefficient quiver $T$ is
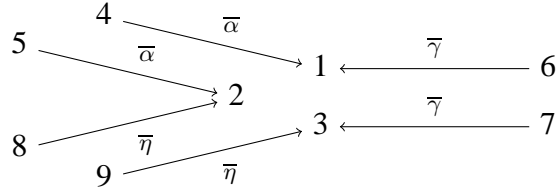
\[
  \begin{tikzpicture}[description/.style={fill=white,inner sep=0pt}]
   \matrix (m) [matrix of math nodes, row sep=-0.2em, column sep=1.5em, text height=1ex, text depth=0ex]
    {   & 4 &   &   &   &   &   &   &   \\
      5 &   &   &   &   &   &   &   &   \\     
        &   &   &   &   &   &   &   &   \\     
        &   &   &   & 1 &   &   &   & 6  \\     
        &   &   & 2 &   &   &   &   &   \\     
        &   &   &   & 3 &   &   &   & 7  \\     
        &   &   &   &   &   &   &   &   \\     
      8 &   &   &   &   &   &   &   &   \\     
        & 9 &   &   &   &   &   &   &   \\ };
    \path[->,font=\scriptsize]
    (m-1-2) edge node[auto] {$\overline\alpha$} (m-4-5)
    (m-2-1) edge node[auto] {$\overline\alpha$} (m-5-4)
    (m-4-9) edge node[auto,swap] {$\overline\gamma$} (m-4-5)
    (m-6-9) edge node[auto,swap] {$\overline\gamma$} (m-6-5)
    (m-8-1) edge node[auto,swap] {$\overline\eta$} (m-5-4)
    (m-9-2) edge node[auto,swap] {$\overline\eta$} (m-6-5);
  \end{tikzpicture}
\]
 where we label the arrows by its image under $F$. It is clear from this picture that $\cB$ is an ordered polarization, and it is easily verified that every relevant pair is maximal for at most one arrow. Thus Theorem \ref{thm: main theorem} implies that the non-empty Schubert cells are affine spaces and that they are indexed by the extremal successor closed subsets $\beta$ of $T_0$. For type $\ue=(2,1,1,1)$, we obtain the non-empty Schubert cells
 \begin{align*}
  C_{\{1,2,4,6,8\}}^M &\simeq \A^0,      & C_{\{1,2,5,6,8\}}^M &\simeq \A^1,      & C_{\{1,3,4,6,9\}}^M &\simeq \A^1, \\
  C_{\{1,3,4,7,9\}}^M &\simeq \A^1,      & C_{\{2,3,5,7,8\}}^M &\simeq \A^1,      & C_{\{2,3,5,7,9\}}^M &\simeq \A^2. 
 \end{align*}
 Therefore the Euler characteristic of $X=\Gr_\ue(M)$ is $6$ and since $X$ is smooth (as we will see in a moment), Corollary \ref{cor: schubert cells form a basis of the cohomology ring} tell us that $H^0(X)=\Z$, $H^1(X)=\Z^4$ and $H^2(X)=\Z$ are additively generated by the closures of the Schubert cells.

 To show that $X$ is smooth, we consider $X$ as a closed subvariety of $\Gr(2,3)\times\P^1\times\P^1\times\P^1$. Note that for a subrepresentation $N$ of $M$ with dimension vector $\ue$, the $1$-dimensional subspaces $N_x$, $N_y$ and $N_z$ of $M_x$, $M_y$ and $M_z$, respectively, determine the $2$-dimensional subspace $N_t$ of $M_t$ uniquely. The images of $N_x=\langle\binom{x_0}{x_1}\rangle$, $N_y=\langle\binom{y_0}{y_1}\rangle$ and $N_z=\langle\binom{z_0}{z_1}\rangle$ in $M_t$ lie in a plane if and only if 
 \[
  \det\left[\begin{smallmatrix}x_0&y_0&0\\x_1&0&z_0\\0&y_1&z_1\end{smallmatrix}\right] \quad = \quad - \, x_0 \, y_1 \, z_0 \, - \, x_1 \, y_0 \, z_1 \quad = \quad 0.
 \]
 Therefore the projection $\Gr(2,3)\times\P^1\times\P^1\times\P^1\to\P^1\times\P^1\times\P^1$ yields an isomorphism
 \[
  \Gr_\ue(M) \quad \stackrel\sim\longrightarrow \quad \bigl\{ \ [ \, x_0:x_1 \, | \, y_0:y_1 \, | \, z_0:z_1 \, ]\in\P^1\times\P^1\times\P^1 \ \bigl| \ x_0 \, y_1 \, z_0 \, + \, x_1 \, y_0 \, z_1  \, = \, 0 \ \bigr\}.
 \]
 Since there is no point in $\Gr_\ue(M)$ for that all derivatives of the defining equation vanishes, $\Gr_\ue(M)$ is smooth.

 The projection $\pi_{1,3}:\P^1\times\P^1\times\P^1\to\P^1\times\P^1$ to the first and third coordinate restricts to a surjective morphism $\pi_{1,3}:\Gr_\ue(M)\to\P^1\times\P^1$. It is bijective outside the fibres of $[1:0|0:1]$ and $[0:1|1:0]$, and these two fibres are
 \[
  \pi_{1,3}^{-1}\bigl([\,1:0\,|\,0:1\,]\bigr) \quad = \quad \bigl\{ [\,1:0\,|\,y_0:y_1\,|\,0:1\,] \ \bigl\} \quad \simeq \quad \P^1 \]and\[ \quad \pi_{1,3}^{-1}\bigl([\,0:1\,|\,1:0\,]\bigr)\quad = \quad \bigl\{ [\,0:1\,|\,y_0:y_1\,|\,1:0\,] \ \bigl\} \quad \simeq \quad \P^1.
 \]

 This shows that $\Gr_\ue(M)$ is the blow-up of $\P^1\times\P^1$ in two points, which is a del Pezzo surface of degree $6$. Note that the closure of the Schubert cells $C_{\{1,2,5,6,8\}}^M$, $C_{\{1,3,4,6,9\}}^M$, $C_{\{1,3,4,7,9\}}^M$ and $C_{\{2,3,5,7,8\}}^M$ are four of the six curves on $\Gr_\ue(M)$ with self-intersection $-1$. In particular, the closures of the latter two cells are the two connected components of the exceptional divisor w.r.t.\ the blow-up $\pi_{1,3}:\Gr_\ue(M)\to\P^1\times\P^1$.

 To return to the opening remark of this example, we see that every point of $\Gr_\ue(M)$, but the intersection points of pairs of $(-1)$-curves, is a subrepresentation of $M$ that is isomorphic to the representation of Example \ref{ex: ramified d_4}. There are six intersection points of pairs of $(-1)$-curves on $\Gr_\ue(M)$, whose coordinates in $\P^1\times\P^1\times\P^1$ are
 \begin{align*}
  & [\,1:0\,|\,1:0\,|\,1:0\,], && [\,0:1\,|\,1:0\,|\,1:0\,], && [\,0:1\,|\,0:1\,|\,1:0\,], \\ & [\,0:1\,|\,0:1\,|\,0:1\,], && [\,1:0\,|\,0:1\,|\,0:1\,], && [\,1:0\,|\,1:0\,|\,0:1\,].
 \end{align*}
 Note that each Schubert cell contains precisely one of these points, and that these points coincide with the subrepresentations $N$ of $M$ that are spanned by the successor closed subsets $\beta$ of $\cB$.

 This exemplifies the idea that the Euler characteristic of a projective variety should equal the number of $\Fun$-points. The naive definition of the $\Fun$-points as the points with coordinates in $\Fun=\{0,1\}$ yields the right outcome in this case. The more elaborate definition of the $\Fun$-points as the Weyl extension $\cW(X_\Fun)$ of the blue scheme $X_\Fun$ associated with $X=\Gr_\ue(M)$ and $\cB$ yields a intrinsic bijection between the elements of $\cW(X_\Fun)$ and the above points. This definition of $\Fun$-points generalizes the connection between Euler characteristics and $\Fun$-points to a larger class of quiver Grassmannians than the naive definition. See \cite[Section 4]{L13} for more details.
\end{ex}

\subsection{Limiting examples}
\label{subsection: limiting examples}

As already mentioned in Remark \ref{rem: different orderings}, there are different possible choices to order $\Rel^2$, which might lead to different generalities of analogues of Theorem \ref{thm: main theorem}. The following examples show, however, that we cannot simply drop an assumption in Theorem \ref{thm: main theorem}.

\begin{ex}[Non-ordered $F$]\label{ex: non-ordered $F$}
  Consider the representation $M=\big[\tinymat 0110:\C^2\to\C^2\big]$ of the quiver $Q=\big[\bullet \to \bullet\big]$. With the obvious choice of ordered basis $\cB=\{1,2,3,4\}$ of $M$, the coefficient quiver $T=\Gamma(M,\cB)$ looks as follows:
 
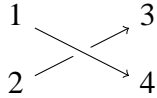
\[
    \begin{tikzpicture}%[description/.style={fill=white,inner sep=0pt}]
   [back line/.style={},cross line/.style={preaction={draw=white, -,line width=6pt}}]
  \matrix (m) [matrix of math nodes, row sep=1em, column sep=1.5em, text height=1ex, text depth=0ex]
   { 1 &   & 3   \\
     2 &   & 4  \\};
   \path[->,font=\scriptsize]
   (m-2-1) edge [back line]  (m-1-3)
   (m-1-1) edge [cross line] (m-2-3);
  \end{tikzpicture}
 \]
 The Schubert cells in the decomposition
 \[
  \Gr_{(1,1)}(M) \quad = \quad C_{\{1,3\}}^M \ \amalg \ C_{\{1,4\}}^M \ \amalg \ C_{\{2,3\}}^M \ \amalg \ C_{\{2,4\}}^M 
 \]
 are easily determined to be
 \[
  C_{\{1,3\}}^M \ = \ \emptyset, \qquad C_{\{1,4\}}^M \ \simeq \ \A^0, \qquad C_{\{2,3\}}^M \ \simeq \ \A^0 \quad\text{and}\quad C_{\{2,4\}}^M \ \simeq \ \Gm.
 \]

 In this examples, we come across a Schubert cell that is isomorphic to $\Gm=\A^1-\A^0$. Theorem \ref{thm: main theorem} does indeed not apply since $F:T\to Q$ is not ordered. However, the other conditions of Theorem \ref{thm: main theorem} are satisfied: $\cB$ is a polarization and every relevant pair is maximal for at most one arrow (since $Q$ has only one arrow).

 Note that the indices of the non-empty Schubert cells are precisely the extremely successor closed subsets $\beta\subset\cB$ of type $\ue$. However, only $\{1,4\}$ and $\{2,3\}$ contribute to the Euler characteristic of $\Gr_\ue(M)\simeq \P^1$, which is $2$. These two subsets are precisely the successor closed subsets of $\cB$, in coherence with the methods of \cite{Cerulli11} and \cite{Haupt12}, which apply to this example.
\end{ex}

\begin{ex}[Non-polarized basis]\label{ex:non-polarized basis}
 Consider the representation $M=\big[\tinymat 1011:\C^2\to\C^2\big]$ of the quiver $Q=\big[\bullet \to \bullet\big]$. With the obvious choice of ordered basis $\cB=\{1,2,3,4\}$ of $M$, the coefficient quiver $T=\Gamma(M,\cB)$ looks as follows:
 \[
    \begin{tikzpicture}[description/.style={fill=white,inner sep=0pt}]
   %[back line/.style={},cross line/.style={preaction={draw=white, -,line width=6pt}}]
  \matrix (m) [matrix of math nodes, row sep=1em, column sep=1.5em, text height=1ex, text depth=0ex]
   { 1 &   & 3   \\
     2 &   & 4  \\};
   \path[->,font=\scriptsize]
   (m-1-1) edge  (m-1-3)
   (m-1-1) edge  (m-2-3)
   (m-2-1) edge  (m-2-3);
  \end{tikzpicture}
 \]
 The Schubert cells in the decomposition
 \[
  \Gr_{(1,1)}(M) \quad = \quad C_{\{1,3\}}^M \ \amalg \ C_{\{1,4\}}^M \ \amalg \ C_{\{2,3\}}^M \ \amalg \ C_{\{2,4\}}^M 
 \]
 are easily determined to be
 \[
  C_{\{1,3\}}^M \ = \ \emptyset, \qquad C_{\{1,4\}}^M \ \simeq \ \A^0, \qquad C_{\{2,3\}}^M \ \simeq \ \A^0 \quad\text{and}\quad C_{\{2,4\}}^M \ \simeq \ \Gm.
 \]

 The Schubert cell $C^M_{\{2,4\}}\simeq\Gm$ does not contradict Theorem \ref{thm: main theorem} since $\cB$ is not a polarization, though the canonical morphism $F:T\to Q$ is ordered and every relevant pair is maximal for at most one arrow (as $Q$ has only one arrow).
\end{ex}

%%%%%%%%%%%%%%%%%%%%%%%%%%%%%%%%%%%%%%%%%%%%%%%%%%%%%%%%%%%%%%%%%%%%%%%%%%%%%%%%%%%%%%%%%%%%%%%%%%%%%%%%%%%%%%%%%%%%%%%%%%%%%%%%%%%%%%%%%%%%%%%%%%%%%%%%%%%%%%%%%%%%%%%%%%%%%%%%%%%%%%%%%%%%%%%%%%%%%%%%%%%%%%%%%%%%%%%%%%%%%%%%%%%%%%%%%%%%%%%%%%%%%%%%%%%%%%%%%%%%%%%%%%%%%%%%%%%%%%%%%%%%%%%%%%%%%%%%%%%%%%%%%%%%%%%%%%%%%%%%%%%%%%%%%%%%%%%%%%%%%%%%%%%%%%%%%%%%%%%%%%%%%%%%%%%%%%%%%%%%%%%%%%%%%%%%%%%%%%%%%%%%%%%%%%%%%%%%%%%%%%

\appendix
\section{Tree modules with polarizations (by Thorsten Weist)}
\label{appendix}

\noindent Let $Q$ be a quiver without loops and oriented cycles. The aim of this appendix is to investigate under which conditions we can construct indecomposable tree modules $X$ such that the basis $\mathcal B$ of the respective coefficient quiver $T_X:=\Gamma(X,\mathcal B)$ is a polarization for $X$. In many cases, the question whether there exists a polarization for $X$ is closely related to the question whether there exists a coefficient quiver without a subdiagram of the form
\[s_1\xrightarrow{a} t_1\xleftarrow{a}s_2\xrightarrow{a} t_2\]
We call a coefficient quiver without such a subdiagram a weak polarization for $X$.
Clearly, a polarization does not have such a subdiagram. But we will see that in many cases these two conditions are already equivalent, for instance for exceptional representations.  In the following, we will not always distinguish between an arrow $a$ of the coefficient quiver and its colour $F(a)$. Moreover, we will often label the arrows of the coefficient quiver by its colour. 

One of the main tools which can be used to construct tree modules is Schofield induction, see \cite{sc2} and \cite{wei} for an application to tree modules. A direct consequence is that, fixing an exceptional sequence $(Y,X)$ with $\Hom(X,Y)=0$
and a basis $(e_1,\ldots,e_m)$ of $\Ext(X,Y)$, representations appearing as the middle terms of exact sequences 
\[\ses{Y^e}{Z}{X^d}\]
give rise to a full subcategory $\mathcal{F}(X,Y)$ of $\mathrm{Rep}(Q)$, the category of representations of $Q$. Moreover, we obtain that $\mathcal{F}(X,Y)$ is equivalent to the category of representations of the generalized Kronecker quiver $K(m)$ with $K(m)_0=\{q_0,q_1\}$ and $K(m)_1=\{\rho_i:q_0\rightarrow q_1\mid i\in\{1,\ldots,m\}\}$ where $m=\dim\Ext(X,Y)$. 
Fixing a real root $\alpha$ of $Q$, we denote by $X_{\alpha}$ the indecomposable representation of dimension $\alpha$, which is unique up to isomorphism. By Schofield induction, we also know that, if $\alpha$ is an exceptional root of $Q$, there already exist exceptional roots $\beta$ and $\gamma$ such that $X_{\beta}\in X_{\gamma}^{\perp}$, $\Hom(X_{\beta},X_{\gamma})=0$ and $\alpha=\beta^d+\gamma^e$ where $(d,e)$ is a real root of the generalized Kronecker quiver $K(\dim\Ext(X_{\beta},X_{\gamma}))$. 

Let $X$ and $Y$ be two representations of a quiver $Q$. Then we can consider the linear map
\[\gamma_{X,Y}:\bigoplus_{q\in Q_0}\Hom_k(X_q,Y_q)\rightarrow\bigoplus_{a:s\rightarrow t\in Q_1}\Hom_k(X_s,Y_t)\]
defined by $\gamma_{X,Y}((f_q)_{q\in Q_0})=(Y_{a}f_s-f_tX_{a})_{a:s\rightarrow t\in Q_1}$.

It is well-known that we have $\ker(\gamma_{X,Y})=\Hom(X,Y)$ and $\mathrm{coker}(\gamma_{X,Y})=\Ext(X,Y)$. The first statement is straightforward. The second statement follows because every morphism $f\in\bigoplus_{a:s\rightarrow t\in Q_1}\Hom_k(X_s,Y_t)$ defines an exact sequence
\[\ses{Y}{((Y_q\oplus X_q)_{q\in Q_0},(\begin{pmatrix}Y_{a}&f_{a}\\0&X_{a}\end{pmatrix})_{a\in Q_1})}{X}\]
with the canonical inclusion on the left hand side and the canonical projection on the right hand side. 

Assume that the representations $X$ and $Y$ are tree modules and let $T_X=\Gamma(X,\mathcal B_X)$ and $T_Y=\Gamma(Y,\mathcal B_Y)$ be the corresponding coefficient quivers. Let $x=\underline\dim X$, $y=\underline\dim Y$. 
Fixing a vertex $q$, from now on we will denote the corresponding vertices of the coefficient quivers by $(\mathcal B_X)_q=\{b_1^q,\ldots,b_{x_q}^q\}$ and $(\mathcal B_Y)_q=\{c_1^q,\ldots,c_{y_q}^q\}$. Let $e^{a}_{k,l}$, where $a:s\rightarrow t\in Q_1$, $k=1,\ldots,x_s$ and $l=1,\ldots,y_t$, be the canonical basis of $\bigoplus_{a:s\rightarrow t}\Hom_k(X_s,Y_t)$ with respect to $\mathcal B_X$ and $\mathcal B_Y$, i.e. $e^{a}_{k,l}(b^s_i)=\delta_{i,k}\delta_{j,l} c_j^t$.% and $e^{a}_{k,l}\mid_{X_q}=0$ for $q\neq s$. 

This means that the coefficient quiver $\Gamma(Z,\mathcal B_X\cup\mathcal B_Y)$ of the middle-term of the exact sequence 
$$E(e^{a}_{k,l}):\ses{Y}{Z}{X}$$ is obtained by adding an extra arrow with colour $a$ from $b^s_k$ to $c^t_l$ to $T_X\cup T_Y$.

Following \cite{wei} we call a basis of $\mathcal E(X,Y)$ of $\Ext(X,Y)$, which solely consists of elements of the form $\overline{e_{i_k,j_k}^{a_k}}$ with $k=1,\ldots,\dim\Ext(X,Y)$, $a_k\in Q_1$, $1\leq i_k\leq x_s$ and $1\leq j_k\leq y_t$, tree-shaped. In abuse of notation, we will not always distinguish between $\overline{e_{i_k,j_k}^{a_k}}$ and $e_{i_k,j_k}^{a_k}$.

Let $X$ be a tree module. For a vertex $b_i^s$ and an arrow $a:s\rightarrow t\in Q_1$ we define 
\[N(a,b_i^s):=\{b^t_j\in (T_X)_0\mid b^s_i\xrightarrow{a} b^t_j\in (T_X)_1\}.\]
Analogously, we define $N(a,b_i^t)$.
If $T_X$ is a weak polarization for $X$, we say that it is strict if we have for all arrows $\arst\in Q_1$ that $|N(a,b_i^s)|\leq 1$ for all $1\leq i\leq x_s$ or $|N(a,b_i^t)|\leq 1$  for all $1\leq i\leq x_t$. Clearly, a weak polarization which is strict is a polarization as defined in Section \ref{subsection: polarizations}. Note that we can always assume that $\mathcal B$ is ordered.

For a vertex $q$ of $T_X$ let $S(q)=\{F(a)\in Q_1\mid\,a\in (T_X)_1,\, s(a)=q\}$ and $T(q)=\{F(a)\in Q_1\mid\,a\in (T_X)_1,\, t(a)=q\}$.

\begin{lemma}\label{injsur}
Let $X$ be a tree module with coefficient quiver $T_X$ such that for every $a\in Q_1$ we have that the map $X_a$ is of maximal rank. Then $T_X$ is a polarization if and only if $T_X$ is a weak polarization.
\end{lemma}
\begin{proof}
Since $X_a$ is of maximal rank, $X_a$ is either surjective or injective. Thus if, in addition, $T_X$ is a weak polarization, this means that $|N(a,b_i^s)|\leq 1$ for all $1\leq i\leq x_s$ or $|N(a,b_i^t)|\leq 1$ for all $1\leq i\leq x_t$. It follows that $T_X$ is a polarization.
\end{proof}

\begin{rem}\label{reminjsur} For general representations of a fixed dimension, and thus in particular for exceptional representations, it is true that all linear maps appearing are of maximal rank.
\end{rem}
Using the notation from above we introduce the following definition:

\begin{df}\label{defpolarized}
\begin{enumerate}\item Let $X$ and $Y$ be two tree modules with coefficient quivers $T_X$ and $T_Y$. Moreover, let $\mathcal E(X,Y)=(e_{i_k,j_k}^{a_k})_k$ with  $s_k\xrightarrow{a_k} t_k\in Q_1$ be a tree-shaped basis of $\Ext(X,Y)$, i.e. $e_{i_k,j_k}^{a_k}(b_{i_k}^{s_k})=c_{j_k}^{t_k}$. Then we call $\mathcal E(X,Y)$ a polarization if 
\begin{enumerate}
\item\label{a} we have that $a_k\notin S(b_{i_k}^{s_k})$ or $a_k\notin T(c_{j_k}^{t_k})$ for all $k$. 
\item\label{b} if $a_k=a_l$ and $b^{s_k}_{i_k}=b^{s_l}_{i_l}$ (resp. $c^{t_k}_{j_k}=c^{t_l}_{j_l}$) for $k\neq l$, we have $a_k\notin T(c_{j_k}^{t_k})$ (resp. $a_k\notin S(b_{i_k}^{s_k})$).
\item\label{c} for all $b_{i_k}^{s_k}\xrightarrow{a_k} b_j^t\in (T_X)_1$ we have $|N(a_k,b_{j}^t)|=1$ and for all $c_i^s\xrightarrow{a_k} c_{j_k}^{t_k}\in (T_Y)_1$ we have $|N(a_k,c_{i}^s)|=1$.
\end{enumerate}
\item If we have $a_k\notin S(b_{i_k}^{s_k})$ \underline{and} $a_k\notin T(c_{j_k}^{t_k})$ for all $k$ in the first condition and if we also have $a_k\neq a_l$ if $k\neq l$, we say that the basis is a strong polarization.
\end{enumerate}
\end{df}

\begin{rem}
Roughly speaking condition (c) makes sure that $b_{i_k}^{s_k}$ is the only neighbor which is connected to $b_j^t$ by an arrow with colour $a_k$.

Condition (a) means that either $b_{i_k}^{s_k}$ is not the source of an arrow with colour $a_k$ (when only the coefficient quiver $T_X$ is considered) or $c_{j_k}^{t_k}$ is not the target of an arrow with colour $a_k$ (when only the coefficient quiver $T_Y$ is considered). 
In particular, if we have $a_k\notin S(b_{i_k}^{s_k})$ \underline{and} $a_k\notin T(c_{j_k}^{t_k})$ for all $k$ in the first condition, the second and third conditions are clearly satisfied.
\end{rem}

Now we are in a position to state under which conditions an exceptional sequence together with a tree-shaped basis of the Ext-group gives rise to indecomposable representations such that, in addition, there exists a coefficient quiver which is a (weak) polarization:

\begin{thm}\label{polar}
Let $(Y,X)$ be an exceptional sequence (of tree modules) such that the coefficient quivers $T_X$ and $T_Y$ are weak polarizations. Moreover, let $\mathcal E(X,Y)=(e^{a_1}_{i_1,j_1},\ldots,e^{a_m}_{i_m,j_m})$ be a basis of $\Ext(X,Y)$ which is a polarization and let $M$ be an indecomposable tree module of $K(m)$.
\begin{enumerate} 
\item If $T_M$ is unramified, then the induced coefficient quiver $T_Z$ of the middle term $Z$ of the corresponding exact sequence 
\[e_M:\ses{Y^e}{Z}{X^d}\]
is a weak polarization for $Z$. Moreover, $Z$ is indecomposable.
\item If the polarization of the basis is strong and $T_M$ is a weak polarization, then the induced coefficient quiver $T_Z$ of the middle term $Z$ of the corresponding exact sequence 
\[e_M:\ses{Y^e}{Z}{X^d}\]
is a weak polarization for $Z$. Moreover, $Z$ is indecomposable.
\item If $X_a$ is injective (resp. surjective) if and only if $Y_a$ is injective (resp. surjective) for all arrows $a\in Q_1$, then $T_Z$ is a weak polarization if and only if $T_Z$ is a polarization.
\item If $M$, and thus also $Z$, is exceptional, the polarization is strict and thus $T_Z$ is a polarization for $Z$.
\end{enumerate}
\end{thm}

\begin{proof}
By simply counting arrows and vertices of the induced coefficient quiver $T_Z$ it follows that $Z$ is a tree module, see also \cite[Proposition 3.9]{wei}. Moreover, since $M$ is indecomposable, by Schofield induction we know that $Z$ is indecomposable. 

Thus we only need to check that $T_Z$ is a weak polarization for $Z$. We first consider the case when $T_M$ is unramified and $\mathcal E(X,Y)$ not necessarily a strong polarization. Clearly, in this case $T_M$ is a weak polarization for $M$. Moreover, note that, since $\mathcal E(X,Y)$ is a basis, if $a_k=a_l$ for $k\neq l$, we either have $j_k\neq j_l$ or $i_k\neq i_l$.

The coefficient quiver could contradict the polarization property if $b^{s_k}_{i_k}=b^{s_l}_{i_l}$ (resp. $c^{t_k}_{j_k}=c^{t_l}_{j_l}$), $a_k=a_l$ for $l\neq k$ and $a_k\in T(c^{t_k}_{j_k})\cap(T_{Y})_1$  (resp. $a_k\in S(b^{s_k}_{i_k})\cap (T_{X})_1$). But this is not possible because $\mathcal E(X,Y)$ is a polarization. Indeed, this would contradict condition (b).
 
Another possibility for $T_Z$ being no weak polarization is if $T_M$ had a subdiagram
\[b_j^{q_0}\xrightarrow{a_i} b_k^{q_1}\xleftarrow{a_i}b_l^{q_0}\text{ or }b_j^{q_1}\xleftarrow{a_i} b_k^{q_0}\xrightarrow{a_i}b_l^{q_1}\]
for some $i\in\{1,\ldots,m\}$. But since $T_M$ is unramified, this is not possible. 

The last possibility for $T_Z$ being no weak polarization were if the basis would contradict condition (c) of Definition \ref{defpolarized}. 

Next we consider the case if the polarization is strong, the representation $M$ is a weak polarization and the representation is not forced to be unramified. But in this case it is straightforward to check that the induced coefficient quiver is a weak polarization. Indeed, for two basis elements $a_k:b_{i_k}^{s_k}\to c_{j_k}^{t_k}$ and $a_l:b_{i_l}^{s_l}\to c_{j_l}^{t_l}$ with $k\neq l$, we have $a_k\neq a_l$ and, moreover, considering the original coefficient quiver $T_X$ and $T_Y$ we have $|N(a_r,q)|=0$ for $q\in\{b_{i_k}^{s_k},c_{j_k}^{t_k}, b_{i_l}^{s_l}, c_{j_l}^{t_l}\}$ and $r\in\{k,l\}$. Thus all subdiagrams which could prevent $T_Z$ from being a weak polarization are forced to be induced from $T_M$. But since $T_M$ is a weak polarization, this cannot happen. 

The third claim is straightforward because, in general, for an exact sequence $e\in\Ext(X,Y)$ with middle term $Z$, the matrix $Z_a$ is a block matrix with diagonal blocks $X_a$ and $Y_a$ for every arrow.

The last claim follows by Lemma \ref{injsur}, see also Remark \ref{reminjsur}.

\end{proof}
\begin{rem} \begin{enumerate}

\item If we are only interested in (weak) polarizations, we can drop the condition that $X$ and $Y$ are exceptional. But in this case it is far more complicated or even impossible to say anything concerning the indecomposability of $Z$.

\item If $Q$ is of extended Dynkin type and, moreover, $(Y,X)$ is an exceptional sequence, we have $\dim\Ext(X,Y)\leq 2$ because
otherwise there would exist a root $d$ of $Q$ having an $n$-parameter family of indecomposables for $n\geq 2$. Then things become easier because every indecomposable tree module of $K(2)$ is unramified.
\end{enumerate}
\end{rem}

Let $S(n)$ be the $n$-subspace quiver with vertices $S(n)_0=\{q_0,q_1,\ldots,q_n\}$ and arrows $S(n)_1=\{q_i\xrightarrow{a_i}q_0\mid i=1,\ldots,n\}$. Let us consider two examples:

\begin{ex} First let $n=4$ and consider the exceptional sequence induced by the roots $\alpha=(2,1,1,1,0)$ and $\beta=(0,0,0,0,1)$. Then coefficient quivers of $X_{\alpha}$, $X_{\beta}$ and a basis of $\Ext(X_{\beta},X_{\alpha})$ are for instance given by
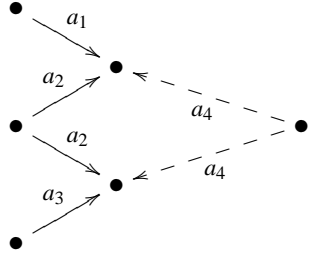
\[
\begin{xy}\xymatrix@R10pt@C25pt{\bullet\ar[rd]^{a_1}\\&\bullet\\\bullet\ar[rd]^{a_2}\ar[ru]^{a_2}&&&\bullet\ar@{-->}[llu]^{a_4}\ar@{-->}[lld]^{a_4}\\&\bullet\\\bullet\ar[ru]^{a_3}
}\end{xy}
\]
Here the dotted arrows correspond to the tree-shaped basis of $\Ext(X_{\beta},X_{\alpha})$ under consideration, whence the remaining vertices and arrows correspond to the two coefficient quivers.

Since the basis of $\Ext(X_{\beta},X_{\alpha})$ is a polarization, which is not strong, and since we have $\dim\Ext(X_{\beta},X_{\alpha})\leq 2$, the first part of Theorem \ref{polar} applies. For instance, considering the exceptional representation of dimension $(1,2)$ of $K(2)$, we obtain
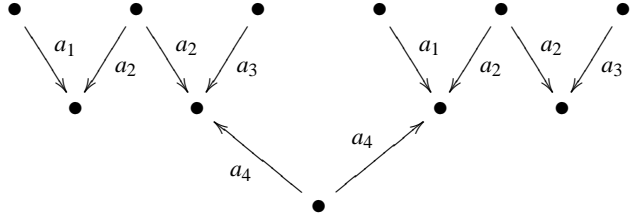
\[
\begin{xy}\xymatrix@R25pt@C10pt{\bullet\ar[rd]^{a_1}&&\bullet\ar[ld]^{a_2}\ar[rd]^{a_2}&&\bullet\ar[ld]^{a_3}&&\bullet\ar[rd]^{a_1}&&\bullet\ar[ld]^{a_2}\ar[rd]^{a_2}&&\bullet\ar[ld]^{a_3}\\&\bullet&&\bullet&&&&\bullet&&\bullet\\&&&&&\bullet\ar[llu]^{a_4}\ar[rru]^{a_4}}\end{xy}
\]
on the $S(4)$-side. This is obviously a (strict) polarization.
\end{ex}

\begin{ex}
An example for a basis which is a strong polarization can be obtained when considering $S(n)$ with $n\geq 3$ and the exceptional sequence induced by the roots $\alpha=(1,1,0,\ldots,0)$ and $\beta=(1,0,1,\ldots,1)$. In this case such a basis of $\Ext(X_{\beta},X_{\alpha})$ is given by choosing $n-2$ out of the $n-1$ maps mapping the one-dimensional subspace $(X_{\beta})_{q_i}$ to $(X_{\alpha})_{q_0}$ for $i=2,\ldots,n$.
\end{ex}

\begin{small}
 \bibliographystyle{plain}
 
\end{small}

\end{document}